\documentclass[11pt]{article}

\usepackage[latin1]{inputenc}
\usepackage[english]{babel}
\usepackage{here,enumerate,graphicx}
\usepackage{amsmath,amssymb,amscd,amsthm,mathrsfs}

\usepackage[flushmargin]{footmisc}
\newcommand\blfootnote[1]{%
	\begingroup
	\renewcommand\thefootnote{}\footnote{#1}%
	\addtocounter{footnote}{-1}%
	\endgroup}

\theoremstyle{plain}
\newtheorem{theorem}{Theorem}[section]
\newtheorem{proposition}[theorem]{Proposition}

\newtheorem{lemma}[theorem]{Lemma}
\newtheorem{corollary}[theorem]{Corollary}

\newtheorem{claim}{Claim}

\theoremstyle{definition}
\newtheorem{definition}[theorem]{Definition}
\newtheorem{remark}[theorem]{Remark}
\newtheorem{example}[theorem]{Example}

\newtheorem{problem}[theorem]{Problem}
\newtheorem{problems}[theorem]{Problems}

\newcommand{\NN}{\mathbb{N}}
\newcommand{\RR}{\mathbb{R}}
\newcommand{\CC}{\mathbb{C}}
\newcommand{\KK}{\mathbb{K}}

\newcommand{\AP}{\mathcal{AP}}
\newcommand{\Bi}{\mathscr{B}}
\newcommand{\BL}{\mathcal{BL}}
\newcommand{\Fc}{\mathcal{F}}
\newcommand{\Gc}{\mathcal{G}}
\newcommand{\Lc}{\mathcal{L}}
\newcommand{\mf}{\mathfrak{m}}
\newcommand{\Nc}{\mathcal{N}}
\newcommand{\Part}{\mathscr{P}}
\newcommand{\Pc}{\mathcal{P}}
\newcommand{\Uc}{\mathcal{U}}

\newcommand{\URec}{\operatorname{URec}}
\newcommand{\FRec}{\operatorname{FRec}}
\newcommand{\RRec}{\operatorname{RRec}}
\newcommand{\Rec}{\textup{Rec}}

\newcommand{\FHC}{\operatorname{FHC}}
\newcommand{\HC}{\textup{HC}}

\newcommand{\lbo}{\ell\mathit{bo}}

\newcommand{\cl}{\overline}
\newcommand{\orb}{\mathcal{O}}
\newcommand{\lspan}{\operatorname{span}}
\newcommand{\eps}{\varepsilon}

\newcommand{\Bdsup}{\overline{\operatorname{Bd}}}
\newcommand{\dinf}{\underline{\operatorname{dens}}}
\newcommand{\BDsup}{\overline{\mathcal{BD}}}
\newcommand{\Dinf}{\underline{\mathcal{D}}}

\def\1{{\mathchoice {\rm 1\mskip-4mu l} {\rm 1\mskip-4mu l} {\rm 1\mskip-4.5mu l} {\rm 1\mskip-5mu l}}}

\newcommand{\supp}{\operatorname{supp}}

\newcommand{\ep}[2]{\langle #1 , #2 \rangle}

\begin{document}
\begin{center}
	\begin{LARGE}
		{\bf Invariant measures from locally bounded orbits}
	\end{LARGE}
\end{center}

\vspace*{-0.85cm}

\begin{center}
	\begin{Large}
		by
	\end{Large}
\end{center}

\vspace*{-0.85cm}

\begin{center}
	\begin{Large}
		Antoni L\'opez-Mart\'inez\blfootnote{\textbf{Journal-ref}: Results in Mathematics, Volume 79, article number 185, (2024).\\ \textbf{DOI}: https://doi.org/10.1007/s00025-024-02204-5}
	\end{Large}
\end{center}

\vspace*{-0.25cm}

\begin{abstract}
	Motivated by recent investigations of Sophie Grivaux and \'Etienne Matheron on the existence of invariant measures in Linear Dynamics, we introduce the concept of {\em locally bounded orbit} for a continuous linear operator $T:X\longrightarrow X$ acting on a Fr\'echet space $X$, and we use this new notion to construct (non-trivial) $T$-invariant probability Borel measures on $(X,\Bi(X))$.
\end{abstract}

\vspace*{-0.75cm}

\section{Introduction}

This paper focusses on some aspects of the relationship between Topological and Measurable Dynamics in the particular context of Linear Dynamics, our main aim being to find some sufficient conditions for a linear dynamical system to admit (non-trivial) invariant probability Borel measures.

A {\em linear dynamical system} is a pair $(X,T)$ where $X$ is a {\em separable infinite-dimensional Fr\'echet space} (that is, a locally convex and completely metrizable topological vector space), and where $T:X\longrightarrow X$ is a {\em continuous linear operator} acting on $X$. We will briefly write $T \in \Lc(X)$, and given a vector $x \in X$ we will denote its {\em $T$-orbit} by
\[
\orb_T(x) := \{ T^nx \ ; \ n \geq 1 \}.
\]

A linear dynamical system $T \in \Lc(X)$ can be examined from various perspectives. For instance, one may focus on {\em Topological Dynamics} and, if we denote by $\cl{E}$ the topological closure of any subset $E \subset X$, one can study notions such as {\em recurrence} and {\em hypercyclicity}: a vector $x \in X$ is said to be
\begin{enumerate}[--]
	\item {\em recurrent for $T$} if $x \in \cl{\orb_T(x)}$, and the {\em set of recurrent vectors for $T$} will be denoted by $\Rec(T)$;
	
	\item {\em hypercyclic for $T$} if $X = \cl{\orb_T(x)}$, and the {\em set of hypercyclic vectors for $T$} will be denoted by $\HC(T)$.
\end{enumerate}
In Linear Dynamics the concept of ``hypercyclicity'' has historically been the main studied property and \cite{BaMa2009_book,GrPe2011_book} represent a comprehensive compilation of such a theory, but ``linear recurrence'' has recently appeared in the 2014 paper \cite{CoMaPa2014}, followed by the works \cite{BoGrLoPe2022}, \cite{CarMur2022_MS}, \cite{CarMur2022_arXiv}, \cite{GriLoPe2022}, \cite{Lopez2022} and \cite{LoMe2023} among others.

Alternatively, one can adopt the {\em Measurable Dynamics} (also called {\em Ergodic Theory}) point of view and, considering a {\em positive finite} (often normalized and therefore {\em probability}) {\em measure} $\mu$ defined on the {\em $\sigma$-algebra of Borel sets} $\Bi(X)$ of $X$, investigate notions such as {\em invariance} and {\em ergodicity}:
\begin{enumerate}[--]
	\item such a measure $\mu$ is called {\em $T$-invariant} (or simply {\em invariant}) if $\mu(A) = \mu(T^{-1}(A))$ for all $A \in \Bi(X)$;
	
	\item and $\mu$ is called {\em $T$-ergodic} (or just {\em ergodic}) if it is invariant and $\mu(A) \in \{0,\mu(X)\}$ when $A=T^{-1}(A)$.
\end{enumerate}
The study of Ergodic Theory in the context of Linear Dynamics started in 1994 with the pioneering work of Flytzanis (see \cite{Flyztanis1994,Flyztanis1995}), and was then further developed in the papers \cite{BaGri2006}, \cite{BaGri2007}, \cite{GriMa2014}, \cite{BaMa2016} and \cite{GriLo2023} between others. See also the textbooks \cite{BaMa2009_book}, \cite{ChoTarVak1987_book} and \cite{Walters1982_book}.

It is by now well understood that the notion of {\em ergodicity} can be seen as the measure-theoretic counterpart of {\em hypercyclicity}, while {\em invariance} can be compared with {\em recurrence}. To state this analogy let $\NN$ be the set of positive integers, denote the {\em return set} from any $x \in X$ to any subset $E \subset X$ by
\[
\Nc_T(x,E) := \{ n \in \NN \ ; \ T^nx \in E \},
\]
and note that a vector $x \in X$ is {\em hypercyclic for $T \in \Lc(X)$} precisely when the return set $\Nc_T(x,U)$ is infinite for every non-empty open subset $U \subset X$, and that $x \in X$ is {\em recurrent for $T$} when $\Nc_T(x,U)$ is infinite at least for every neighbourhood $U$ of $x$. Using this notation we reach the announced analogy:
\begin{enumerate}[--]
	\item when $\mu$ is a $T$-ergodic measure with {\em full support} (that is, $\mu(U)>0$ for every open set $U \neq \varnothing$), it was exhibited by Bayart and Grivaux in 2006 that then $\mu$-a.e.\ vector $x \in X$ is not only hypercyclic, but even {\em frequently hypercyclic}: for every non-empty open subset $U \subset X$ the return set $\Nc_T(x,U)$ has positive lower density $\dinf(\Nc_T(x,U))>0$, where the {\em lower density} for any set $A \subset \NN$ is
	\[
	\dinf(A) := \liminf_{N\to\infty} \frac{\# (A \cap [1,N])}{N};
	\]
	the vector $x \in X$ is then called {\em frequently hypercyclic for $T$} and we will denote by $\FHC(T)$ the {\em set of frequently hypercyclic vectors for $T$}; see \cite[Proposition~3.12]{BaGri2006} or \cite[Corollary~5.5]{BaMa2009_book} for the details of this argument, which uses the {\em Birkhoff pointwise ergodic theorem} in a crucial way;
	
	\item and when $\mu$ is just $T$-invariant, it was checked in \cite{GriLo2023} that then $\mu$-a.e.\ vector $x \in X$ is also not only recurrent, but even {\em frequently recurrent}: for every neighbourhood $U$ of $x$ the return set $\Nc_T(x,U)$ has positive lower density $\dinf(\Nc_T(x,U))>0$; the vector $x \in X$ is then called {\em frequently recurrent for~$T$} and we will denote by $\FRec(T)$ the {\em set of frequently recurrent vectors for $T$}; see \cite[Lemma~3.1]{GriLo2023} for the details of this argument, which uses again the {\em Birkhoff pointwise ergodic theorem} this time combined with the {\em ergodic decomposition theorem}, and see \cite{BoGrLoPe2022} for more on frequent recurrence.
\end{enumerate}

These results emphasize the importance of being able to ensure the existence of invariant measures possibly satisfying additional properties such as having full support or being ergodic, weakly and even strongly mixing. This kind of question goes back to the classical work of Oxtoby and Ulam \cite{OxUlam1939} where the existence of invariant positive finite Borel measures, but for continuous automorphisms acting on completely metrizable spaces, was fully characterized. In our linear framework note that every operator $T \in \Lc(X)$ admits the atomic {\em Dirac mass} $\delta_0$ as an invariant measure since the zero-vector is always a fixed point, so we will say that a probability (or positive finite) Borel measure $\mu$ on $(X,\Bi(X))$ is {\em non-trivial} if it differs from $\delta_0$ (or from every positive multiple of $\delta_0$).

The existence of non-trivial invariant measures in Linear Dynamics has recently been explored in the works \cite{GriMa2014} and \cite{GriLo2023}. In fact, \cite[Section~2]{GriMa2014} extends to the linear setting a constructive technique already known for compact dynamical systems obtaining that, under some ``natural topological assumptions'' on the space $X$ and the operator $T$, then one can construct a $T$-invariant measure with full support from a {\em frequently hypercyclic} vector $x \in \FHC(T)$. This was slightly refined in \cite[Section~2]{GriLo2023} weakening the ``frequent hypercyclicity'' requirement into that of ``reiterative recurrence'': we say that $x \in X$ is {\em reiteratively recurrent for $T$} if for every neighbourhood $U$ of $x$ the return set $\Nc_T(x,U)$ has positive upper Banach density $\Bdsup(\Nc_T(x,U))>0$, where the {\em upper Banach density} for any set $A \subset \NN$ is
\[
\Bdsup(A) := \limsup_{N\to\infty} \left( \max_{m \geq 0} \frac{\# (A \cap [m+1,m+N])}{N} \right).
\]
We will denote by $\RRec(T)$ the {\em set of reiteratively recurrent vectors for $T$} and, even though the inclusions $\FHC(T) \subset \FRec(T) \subset \RRec(T)$ are usually strict (see \cite[Theorem~5.8]{BoGrLoPe2022}), exactly the same ``natural topological assumptions'' from \cite[Section~2]{GriMa2014} were used in \cite[Section~2]{GriLo2023} to construct a non-trivial $T$-invariant measure from each non-zero reiteratively recurrent vector $x \in \RRec(T) \setminus \{0\}$.\newpage

The aforementioned ``topological assumptions'' on $T \in \Lc(X)$ require the underlying space $X$ to be a {\em Banach space} in both works \cite{GriMa2014}~and~\cite{GriLo2023}, since some kind of ``local boundedness'' is needed along the construction of invariant measures developed. The main objective of this paper, and what we do in Section~\ref{Sec_2:invariant.measures}, is extending the constructive technique exposed in \cite{GriMa2014,GriLo2023} to the context of operators acting on Fr\'echet spaces via the new concept of {\em locally bounded orbit} (see Definition~\ref{Def:lbo}). The rest of the paper is organized as follows: in Section~\ref{Sec_3:applications} we apply the theory developed in Section~\ref{Sec_2:invariant.measures} by adding general restrictions on $X$ and $T$, we discuss why the invariant measures constructed are optimal in terms of Banach limits, and we adapt the main ideas from Section~\ref{Sec_2:invariant.measures} to study {\em almost-$\Fc$-recurrence} and some equivalences of {\em Devaney chaos} in the Fr\'echet setting. In Section~\ref{Sec_4:lbo} we elaborate further on the notion of ``locally bounded orbit'' by exhibiting some explicit examples and stability results.

\section{Invariant measures on Fr\'echet spaces}\label{Sec_2:invariant.measures}

In this section we recall the technique developed in \cite{GriMa2014} and \cite{GriLo2023} to construct invariant measures for operators acting on Banach spaces and we extend it to the Fr\'echet setting by introducing the concept of {\em locally bounded orbit} (see Definition~\ref{Def:lbo} below). The basic results that we need from \cite{GriMa2014,GriLo2023} were originally stated for Polish dynamical systems so that we start by presenting some notation.

\subsection{From the Banach to the Fr\'echet case}

We will say that the pair $(X,T)$ is a {\em Polish dynamical system} if $T:X\longrightarrow X$ is a continuous map acting on a Polish space $X$, that is, a separable completely metrizable topological space. Note that the concept of ``linear dynamical system'' as defined at the Introduction of this paper is indeed a particular case of Polish system. Moreover, the topological and measurable notions already defined, such as ``recurrent/hypercyclic vector'' and ``invariant/ergodic measure'', make sense in this rather general context and by abuse of notation we will utilize them also for Polish systems. See \cite{DaGlas2015} for recent investigations on the relation between both Polish and linear dynamical systems.

Given a Polish space $X$ we will denote by $\tau_X$ the original (separable and completely metrizable) topology of the space, but we will often consider a second topology $\tau$ on $X$ fulfilling some properties with respect to $\tau_X$. The {\em $\sigma$-algebra of Borel sets} induced by each of these topologies will be denoted by $\Bi(X,\tau_X)$ and $\Bi(X,\tau)$ respectively, and if they coincide we will simply write $\Bi(X)$. All the measures considered in this paper will be {\em non-negative finite Borel measures} defined on Polish spaces, hence {\em regular} (see \cite[Proposition~8.1.12]{Cohn2013_book}), and we will usually omit the words ``Borel'' and ``regular''. Moreover, for any non-negative measure $\mu$ on a Polish space $(X,\tau_X)$ we will denote its {\em support} by
\[
\supp(\mu) := X \setminus \bigcup \{ U \subset X \ ; \ U \text{ is $\tau_X$-open and } \mu(U)=0 \}.
\]
It is easy to check that a point $x \in X$ belongs to the support $\supp(\mu)$ if and only if $\mu(U)>0$ for every measurable neighbourhood $U$ of $x$. Let $\ell^{\infty}$ be the {\em space of all bounded sequences of real numbers}, we will write $\1 \in \ell^{\infty}$ for the sequence with all its terms equal to $1$, and for each $A \subset \NN$ the element $\1_A \in \ell^{\infty}$ will be the sequence in which the $n$-th coordinate is exactly $1$ if $n \in A$ and $0$ otherwise. Recall also that a {\em Banach limit} is a positive and shift-invariant continuous linear functional $\mf:\ell^{\infty}\longrightarrow \RR$, which preserves the value of the limit for every convergent sequence (see \cite[page~82]{Conway1989_book}).

Using the previously introduced notation we can explore the very technical lemma, originally stated in \cite[Remarks~2.6~and~2.12]{GriMa2014} and later refined in \cite[Lemma~2.1]{GriLo2023}, which allows to construct plenty of invariant (but possibly null) measures for every Polish dynamical system $T:(X,\tau_X)\longrightarrow (X,\tau_X)$ admitting a second Hausdorff topology $\tau$ on $X$ which fulfills some conditions with respect to $\tau_X$:
\begin{enumerate}[--]
	\item \cite[Lemma~2.1]{GriLo2023}: {\em Let $(X,T)$ be a Polish dynamical system, denote by $\tau_X$ the original topology of $X$ and suppose that there exists a Hausdorff topology $\tau$ on $X$ fulfilling that
		\begin{enumerate}[{\em(a)}]
			\item[{\em($\alpha$)}] $T:X\longrightarrow X$ is $\tau$-$\tau$-continuous;
			
			\item[{\em($\beta$)}] $\tau \subset \tau_X$;
			
			\item[{\em($\gamma$)}] every $\tau$-compact set is $\tau$-metrizable;
			
			\item[{\em($\delta$)}] $\Bi(X,\tau) = \Bi(X,\tau_X)$;
		\end{enumerate}
		then for each $x_0 \in X$ and each Banach limit $\mf:\ell^{\infty}\longrightarrow \RR$ one can find a (non-negative) $T$-invariant finite Borel regular measure $\mu$ on $(X,\Bi(X))$ for which $\mu(X)\leq 1$ and such that $\mu(K) \geq \mf(\1_{\Nc_T(x_0,K)})$ for every $\tau$-compact set $K \subset X$. Moreover, we have the inclusion
		\[
		\supp(\mu) \subset \cl{\orb_T(x_0)}^{\tau}.
		\]}
\end{enumerate}

In \cite[Theorem~2.3]{GriLo2023} it is shown that conditions slightly stronger than ($\alpha$), ($\beta$), ($\gamma$) and ($\delta$) allow to obtain non-null measures by applying \cite[Lemma~2.1]{GriLo2023} to each reiteratively recurrent point. This result has the following automatic corollary (already observed in \cite[Proof of Theorem~1.3]{GriLo2023}):

\begin{corollary}\label{Cor:Banach->measure}
	Let $Y$ be a Banach space, assume that its dual Banach space $X:=Y'$ is separable, and let $T \in \Lc(X)$ be the adjoint of some $S \in \Lc(Y)$. Given a (non-zero) vector $x_0 \in \RRec(T)$ one can find a (non-trivial) $T$-invariant probability measure $\mu_{x_0}$ on $(X,\Bi(X))$ such that
	\[
	x_0 \in \supp(\mu_{x_0}) \subset \cl{\orb_T(x_0)}^{\sigma(X,Y)}.
	\]
	Moreover, if the set $\RRec(T)$ is dense in $X$, then there exists a $T$-invariant probability measure $\mu$ on $(X,\Bi(X))$ with full support. In particular, the result is true for every operator $T \in \Lc(X)$ with respect to the weak topology $\sigma(X,X')$ as soon as $(X,\|\cdot\|)$ is a separable reflexive Banach space.
\end{corollary}

Note that given any Banach space $X$ we are denoting by $X'$ its {\em topological dual space}, which is again a Banach space, and given a dual pair $(Y,X)$ we are denoting by $\sigma(X,Y)$ the {\em weak topology} on the space $X$ induced by $Y$. Using this ``locally convex spaces''-notation let us briefly explain how Corollary~\ref{Cor:Banach->measure} is implicitly proved in \cite[Proof of Theorem~1.3 and Theorem~2.3]{GriLo2023}: when $T \in \Lc(X)$ is the adjoint operator of some $S \in \Lc(Y)$ it is well-known that
\begin{enumerate}[(a)]
	\item[($\alpha$)] $T:X\longrightarrow X$ is $\sigma(X,Y)$-$\sigma(X,Y)$-continuous;
	
	\item[($\beta$)] $\sigma(X,Y) \subset \tau_{\|\cdot\|}$, where $\tau_{\|\cdot\|}$ is the norm topology of $(X,\|\cdot\|)$;
	
	\item[($\gamma^*$)] every vector of $X$ has a basis of $\tau_{\|\cdot\|}$-neighbourhoods consisting of $\sigma(X,Y)$-compact sets;
\end{enumerate}
and if $(X,\|\cdot\|)$ is separable \cite[Fact~2.1]{GriMa2014} shows that condition ($\gamma^*$) implies the $\sigma(X,Y)$-metrizability of every $\sigma(X,Y)$-compact set, but also that $\Bi(X,\sigma(X,Y)) = \Bi(X,\tau_{\|\cdot\|})$, which are conditions ($\gamma$)~and~($\delta$) needed to apply \cite[Lemma~2.1]{GriLo2023}. Then, each reiteratively recurrent vector $x_0 \in \RRec(T)$ can be shown to return enough frequently to every of its $\sigma(X,Y)$-compact and $\tau_{\|\cdot\|}$-neighbourhoods of the type $K_r := \{ x \in X \ ; \ \|x_0-x\|\leq r \}$ for $r>0$, to admit a Banach limit $\mf_r:\ell^{\infty}\longrightarrow \RR$ such that
\[
\mf_r(\1_{\Nc_T(x_0,K_r)}) = \Bdsup(\Nc_T(x_0,K_r)) > 0.
\]
In order to repeat this proof when $(X,\tau_X)$ is a Fr\'echet space obtained as the strong dual of some locally convex space $(Y,\tau_Y)$, we must prove that conditions ($\alpha$), ($\beta$), ($\gamma$) and ($\delta$) still hold between $\sigma(X,Y)$ and $\tau_X$ (see Lemma~\ref{Lem:top.conditions} below), but we also need to solve the following problem: if $(X,\tau_X)$ is not a Banach space then the $\tau_X$-neighbourhoods of the vector selected $x_0 \in \RRec(T)$ are no longer $\sigma(X,Y)$-compact and \cite[Lemma~2.1]{GriLo2023} seems useless. The following definition will avoid this issue:

\begin{definition}\label{Def:lbo}
	Let $T\in \Lc(X)$ be an operator acting on a Fr\'echet space $X$. A vector $x \in X$ has a {\em locally bounded orbit for $T$} if there exists a neighbourhood $U$ of $x$ such that the set $U \cap \orb_T(x)$ is bounded in $X$. We will denote by $\lbo(T)$ the {\em set of vectors with locally bounded orbit for $T$}.
\end{definition}

Using this new concept we can state Theorem~\ref{The:lbo->measure} below, which is the main result of this paper. Recall first that given any Hausdorff locally convex topological vector space $(Y,\tau_Y)$, then its topological dual space $Y'$ can be endowed with a Hausdorff locally convex topology $\beta(Y',Y)$ for which a basis of $\beta(Y',Y)$-neighbourhoods of the $0_{Y'}$-vector is formed by the following family of $\sigma(Y',Y)$-closed sets
\[
\left\{ E^{\circ} \ ; \ E \subset Y \text{ is a bounded subset in } (Y,\tau_Y) \right\},
\]
where $E^{\circ} := \{ x \in Y' \ ; \ |\ep{u}{x}| \leq 1 \text{ for all } u \in E \} \subset Y'$ denotes the ({\em absolute}) {\em polar} of each set $E \subset Y$ with respect to the dual pair $(Y,Y')$. The topology $\beta(Y',Y)$ is called the {\em strong topology} on the space $Y'$ induced by $(Y,\tau_Y)$, and the Hausdorff locally convex topological vector space $(Y',\beta(Y',Y))$ is called the {\em strong dual} of $(Y,\tau_Y)$; see \cite[Chapter~8]{Jarchow1981_book} for more on duality for locally convex spaces. Recall also that $(Y,\tau_Y)$ is called {\em quasi-$\ell_{\infty}$-barrelled} if every bounded sequence in its strong dual space is equicontinuous (see \cite[Section~12.1]{Jarchow1981_book} or \cite[Definition~8.2.13]{PeBo1987_book}). Here we have our main result:

\begin{theorem}\label{The:lbo->measure}
	Let $(Y,\tau_Y)$ be a quasi-$\ell_{\infty}$-barrelled Hausdorff locally convex topological vector space, assume that its strong dual $(X,\tau_X):=(Y',\beta(Y',Y))$ is a separable Fr\'echet space, and let $T \in \Lc(X)$ be the adjoint of some linear map $S:Y\longrightarrow Y$. Given a (non-zero) vector $x_0 \in \RRec(T) \cap \lbo(T)$ one can find a (non-trivial) $T$-invariant probability measure $\mu_{x_0}$ on $(X,\Bi(X))$ such that
	\[
	x_0 \in \supp(\mu_{x_0}) \subset \cl{\orb_T(x_0)}^{\sigma(X,Y)}.
	\]
	Moreover, if the set $\RRec(T) \cap \lbo(T)$ is dense in $X$, then there exits a $T$-invariant probability measure $\mu$ on $(X,\Bi(X))$ with full support. In particular, the result is true for every operator $T \in \Lc(X)$ with respect to the weak topology $\sigma(X,X')$ as soon as $(X,\tau_X)$ is a separable reflexive Fr\'echet space.
\end{theorem}

The rest of this section is devoted to prove Theorem~\ref{The:lbo->measure}, but let us include some initial remarks:

\begin{remark}\label{Rem:lbo.1}
	Let $T \in \Lc(X)$ be an operator acting on a Fr\'echet space $X$. Note that:
	\begin{enumerate}[(a)]
		\item When $X$ is Banach the equality $\lbo(T)=X$ holds because the unit ball of $X$ is a bounded set, so that Theorem~\ref{The:lbo->measure} is just an extension of Corollary~\ref{Cor:Banach->measure} to operators acting on Fr\'echet spaces.
		
		\item When $X$ is a Fr\'echet space which is not Banach:
		\begin{enumerate}[(b1)]
			\item We have that $X\setminus\Rec(T) \subset \lbo(T)$. Indeed, given any $x \in X\setminus\Rec(T)$ there is some neighbourhood $U$ of $x$ such that $U \cap \orb_T(x)$ is finite, so that $X\setminus\lbo(T) \subset \Rec(T)$.
			
			\item If $x \in \Rec(T)$ has a {\em bounded orbit for $T$} (that is, the set $\orb_T(x)$ is bounded in $X$) then $x \in \lbo(T)$. In particular: if the vector $x$ is {\em $T$-periodic} (that is, $T^px=x$ for some $p \in \NN$) or if $x$ is a {\em unimodular $T$-eigenvector} (that is, $Tx=\lambda x$ with $|\lambda|=1$) then $x \in \lbo(T)$; and if the operator $T$ is {\em power-bounded} then $\lbo(T)=X$ (see Subsection~\ref{SubSec_3.1:applications+optimality} and Section~\ref{Sec_4:lbo}).
			
			\item We have that $\HC(T) \subset \Rec(T) \setminus \lbo(T)$. Indeed, given $x \in \HC(T)$ and any neighbourhood $U$ of $x$ then $U \cap \orb_T(x)$ is dense in $U$ and not bounded since $X$ is not Banach. Thus, if $T$ is {\em Devaney chaotic} (that is, $T$ has a hypercyclic vector and the $T$-periodic vectors are dense) then $\lbo(T)$ is a dense but meager set in $X$ (see Subsection~\ref{SubSec_3.2:PF+chaos} and Section~\ref{Sec_4:lbo}).
		\end{enumerate}
	\end{enumerate}
	See Section~\ref{Sec_4:lbo} for more on this new concept of {\em locally bounded orbit}.
\end{remark}

\begin{remark}\label{Rem:Theorem.conditions}
	The reader is referred to the textbooks \cite{Jarchow1981_book,PeBo1987_book} for details regarding the following facts:
	\begin{enumerate}[(a)]
		\item The space $(Y,\tau_Y)$ in Theorem~\ref{The:lbo->measure} has to be a separable quasi-barrelled (DF)-space. Indeed, since the strong dual space $(Y',\beta(Y',Y))$ is assumed to be separable we deduce that $(Y,\tau_Y)$ has to be separable, and hence quasi-barrelled by \cite[Corollary~8.2.20]{PeBo1987_book}, but we also know that $(Y',\beta(Y',Y))$ is a Fr\'echet space so that $(Y,\tau_Y)$ has a fundamental sequence of bounded sets (see \cite[Corollary~5]{BiersBo1992}).
		
		\item Conversely to (a), and since the strong dual of any (DF)-space is always a Fr\'echet space (see for instance \cite[Section~12.4]{Jarchow1981_book}), we have that the hypothesis of Theorem~\ref{The:lbo->measure} are satisfied as soon as the starting space $(Y,\tau_Y)$ is a (DF)-space with separable strong dual.
		
		\item Note that, when $(X,\tau_X):=(Y',\beta(Y',Y))$, the definition of strong topology implies that every vector $x \in X$ admits a basis of $\tau_X$-neighbourhoods formed by $\sigma(X,Y)$-closed sets.
		
		\item In the statement of Theorem~\ref{The:lbo->measure} the sentence ``{\em $T \in \Lc(X)$ is the adjoint of $S:Y\longrightarrow Y$}'' means that ``{\em we have the dual-evaluation equality $\ep{Su}{x}=\ep{u}{Tx}$ for every pair $(u,x) \in Y\times X$}''.
	\end{enumerate}
\end{remark}

We are now ready to prove Theorem~\ref{The:lbo->measure}. See Subsection~\ref{SubSec_2.3:remarks} for some examples and extra remarks.

\subsection{Proof of Theorem~\ref{The:lbo->measure}}\label{SubSec_2.2:proof}

Let us start by showing that \cite[Lemma~2.1]{GriLo2023} can be used in our Fr\'echet setting. Recall first that a topological space is called {\em Lindel\"of} if every open cover of the space admits a countable subcover, and that a topological space is called {\em hereditarily Lindel\"of} if every subspace of it is Lindel\"of.

\begin{lemma}\label{Lem:top.conditions}
	Let $(Y,\tau_Y)$ be a Hausdorff locally convex topological vector space, denote its strong dual space by $(X,\tau_X):=(Y',\beta(Y',Y))$, and let $T:X\longrightarrow X$ be a linear map. Then:
	\begin{enumerate}[{\em(a)}]
		\item[{\em($\alpha$)}] $T$ is $\sigma(X,Y)$-$\sigma(X,Y)$-continuous if and only if it is the adjoint of some linear map $S:Y\longrightarrow Y$;
		
		\item[{\em($\beta$)}] $\sigma(X,Y) \subset \tau_X$.
	\end{enumerate}
	Moreover, if the Hausdorff locally convex space $(Y,\tau_Y)$ is separable then:
	\begin{enumerate}[{\em(a)}]
		\item[{\em($\gamma$)}] every $\sigma(X,Y)$-compact set is $\sigma(X,Y)$-metrizable;
	\end{enumerate}
	and if the Hausdorff locally convex space $(X,\tau_X)$ is hereditarily Lindel\"of then:
	\begin{enumerate}[{\em(a)}]
		\item[{\em($\delta$)}] $\Bi(X,\sigma(X,Y))=\Bi(X,\tau_X)$.
	\end{enumerate}
\end{lemma}
\begin{proof}
	Property ($\alpha$) is well-known (see \cite[Section~8.6]{Jarchow1981_book}) and ($\beta$) follows from the definition of $\sigma(X,Y)$ and $\tau_X$. Property ($\gamma$) is also known since for any $\tau_Y$-dense countable set $\{u_k : k \in\NN\} \subset Y$ the function
	\[
	\textstyle d(x,y) := \sum_{k\in\NN} 2^{-k} \cdot \min\{1,|\ep{u_k}{x-y}|\} \quad \text{ for each } (x,y) \in X\times X,
	\]
	can be checked to be a metric defining the weak topology on each $\sigma(X,Y)$-compact subset. For ($\delta$) note first that $\Bi(X,\sigma(X,Y)) \subset \Bi(X,\tau_X)$ by ($\beta$). Conversely, let $U \in \tau_X$ and for each $x \in U$ select a $\sigma(X,Y)$-closed $\tau_X$-neighbourhood $U_x$ of $x$ such that $U_x \subset U$. Thus, $U = \bigcup_{x \in U} U_x$ is a covering of $U$. If now $(X,\tau_X)$ is hereditarily Lindel\"of we can obtain a countable sub-covering of $U$ formed by $\sigma(X,Y)$-closed sets, which finally implies that $U \in \Bi(X,\sigma(X,Y))$.
\end{proof}

We can now proceed with the proof of Theorem~\ref{The:lbo->measure}: let $(Y,\tau_Y)$ be a quasi-$\ell_{\infty}$-barrelled Hausdorff locally convex topological vector space, assume that its strong dual space $(X,\tau_X):=(Y',\beta(Y',Y))$ is a separable Fr\'echet space, and let $T \in \Lc(X)$ be the adjoint of some linear map $S:Y\longrightarrow Y$.

Since the space $(X,\tau_X)$ is assumed to be separable and metrizable one can check that $(Y,\tau_Y)$ is also separable, and that $(X,\tau_X)$ is hereditarily Lindel\"of. Hence Lemma~\ref{Lem:top.conditions} applies to our situation in its full generality and we have that:

\begin{claim}\label{Claim_1}
	Given $x_0 \in X$ and a $\sigma(X,Y)$-compact set $K \subset X$ with $\Bdsup(\Nc_T(x_0,K))>0$, there exists a $T$-invariant probability measure $\mu$ on $(X,\Bi(X))$ such that $\mu(K)>0$. Moreover, we have the inclusion
	\[
	\supp(\mu) \subset \cl{\orb_T(x_0)}^{\sigma(X,Y)}.
	\]
\end{claim}
\begin{proof}
	In \cite[Fact~2.3.1]{GriLo2023} it was shown that for any set $A \subset \NN$ one can construct a Banach limit $\mf_A:\ell^{\infty}\longrightarrow \RR$ such that $\mf_A(\1_A)=\Bdsup(A)$. We repeat the explicit construction of such a Banach limit in Proposition~\ref{Pro:BDsup=BL} below, where we discuss the optimality of this construction (see Subsection~\ref{SubSec_3.1:applications+optimality}). Hence there exists a Banach limit $\mf_K:\ell^{\infty}\longrightarrow \RR$ such that $\mf_K(\1_{\Nc_T(x_0,K)}) = \Bdsup(\Nc_T(x_0,K)) > 0$ for the $\sigma(X,Y)$-compact set $K$. By Lemma~\ref{Lem:top.conditions} we can now apply \cite[Lemma~2.1]{GriLo2023} to $x_0$ and $\mf_K$ obtaining a $T$-invariant positive finite measure $\mu$ on $(X,\Bi(X))$ for which $\mu(K) \geq \mf_K(\1_{\Nc_T(x_0,K)}) > 0$.
\end{proof}

Now we use the ``locally bounded orbit''-assumption:

\begin{claim}\label{Claim_2}
	Given a vector $x_0 \in \RRec(T) \cap \lbo(T)$ there exists a $T$-invariant probability measure $\mu_{x_0}$ on $(X,\Bi(X))$ such that
	\[
	x_0 \in \supp(\mu_{x_0}) \subset \cl{\orb_T(x_0)}^{\sigma(X,Y)}.
	\]
	In particular, if $x_0\neq 0$ then $\mu_{x_0}$ is a non-trivial $T$-invariant measure.
\end{claim}
\begin{proof}
	Since the vector $x_0$ has a locally bounded orbit for $T$ there exists a $\tau_X$-neighbourhood $U$ of $x_0$ such that $K := U \cap \orb_T(x_0)$ is a countable $\tau_X$-bounded set and hence an equicontinuous set in $Y'$ by the quasi-$\ell_{\infty}$-barrelled assumption on $(Y,\tau_Y)$. Moreover, we have that $K = U \cap \orb_T(x_0)$ is relatively $\sigma(X,Y)$-compact by the Alaoglu-Bourbaki theorem (see \cite[Section~8.5, Theorem~2]{Jarchow1981_book}).
	
	Let $(U_n)_{n\in\NN}$ be a basis of $\sigma(X,Y)$-closed $\tau_X$-neighbourhoods of $x_0$ and let $(V_n)_{n\in\NN} := (U_n \cap U)_{n\in\NN}$, which is again a basis of $\tau_X$-neighbourhoods of $x_0$. Set
	\[
	K_n := \cl{V_n \cap \orb_T(x_0)}^{\sigma(X,Y)} = \ \cl{U_n \cap K}^{\sigma(X,Y)} \subset \ U_n \cap \cl{K}^{\sigma(X,Y)} \quad \text{ for each } n \in \NN.
	\]
	Note that every $K_n$ is a $\sigma(X,Y)$-compact set included in $U_n$ for which
	\[
	\Nc_T(x_0,V_n) = \Nc_T(x_0,V_n \cap \orb_T(x_0)) \subset \Nc_T(x_0,K_n) \subset \Nc_T(x_0,U_n).
	\]
	Hence we can apply \textbf{Claim~\ref{Claim_1}} to $x_0$ and each set $K_n$ obtaining a sequence $(\mu_n)_{n\in\NN}$ of $T$-invariant probability measures on $(X,\Bi(X))$ for which $\mu_n(U_n) \geq \mu_n(K_n) > 0$ and such that
	\[
	\supp(\mu_n) \subset \cl{\orb_T(x_0)}^{\sigma(X,Y)} \quad \text{ for each } n \in \NN.
	\]
	Then $\mu_{x_0} := \sum_{n\in\NN} 2^{-n} \cdot \mu_n$ is a $T$-invariant probability measure on $(X,\Bi(X))$ fulfilling that
	\[
	x_0 \in \supp(\mu_{x_0}) \subset \cl{\orb_T(x_0)}^{\sigma(X,Y)},
	\]
	which also implies that $\mu_{x_0}$ is non-trivial when $x_0\neq 0$ (see \cite[Fact~2.3.2]{GriLo2023} for more details).
\end{proof}

To prove the part of Theorem~\ref{The:lbo->measure} regarding the existence of a $T$-invariant measure with full support one can argue just as in \cite[Theorem~2.3]{GriLo2023}:

\begin{claim}
	If the set $\RRec(T) \cap \lbo(T)$ is dense in $X$ there exits a $T$-invariant probability measure $\mu$ on $(X,\Bi(X))$ with full support.
\end{claim}
\begin{proof}
	Given any countable dense subset $\{x_k \ ; \ k\in\NN\} \subset \RRec(T) \cap \lbo(T)$ and applying \textbf{Claim~\ref{Claim_2}} to each vector $x_k$ we obtain a sequence $(\mu_{x_k})_{k\in\NN}$ of $T$-invariant probability measures on $(X,\Bi(X))$ such that $x_k \in \supp(\mu_{x_k})$ for each $k \in \NN$, and $\mu := \sum_{k\in\NN} 2^{-k} \cdot \mu_{x_k}$ fulfills the desired properties.
\end{proof}

Finally, and in order to complete the proof of Theorem~\ref{The:lbo->measure}, let us argue what happens for reflexive spaces. Recall first that a Fr\'echet space $(X,\tau_X)$ is called {\em reflexive} if the canonical inclusion from $X$ into its strong bi-dual space $X''$ is an isomorphism, that is, the linear map
\[
J:(X,\tau_X)\longrightarrow (X'',\beta(X'',X'))
\]
where given $x \in X$ the map $J(x) : X' \longrightarrow \KK$ acts as $[J(x)](u) = \ep{J(x)}{u} := \ep{u}{x}$ for all $u \in X'$, and where $\beta(X'',X')$ is the strong topology induced on $X''$ by the locally convex space $(X',\beta(X',X))$. Since the strong dual of a Fr\'echet space is always a (DF)-space, in the previous situation we have that $(Y,\tau_Y) := (X',\beta(X',X))$ is a (DF)-space; see \cite[Section~12.4]{Jarchow1981_book} and \cite[Chapter~23]{MeisVogt1997_book} for more details.

\begin{claim}
	The conclusion of Theorem~\ref{The:lbo->measure} holds for every operator $T \in \Lc(X)$ with respect to the weak topology $\sigma(X,X')$ as soon as $(X,\tau_X)$ is a separable reflexive Fr\'echet space.
\end{claim}
\begin{proof}
	Let $(X,\tau_X)$ be a separable reflexive Fr\'echet space, consider any operator $T \in \Lc(X)$ and denote by $(Y,\tau_Y) := (X',\beta(X',X))$ the strong dual space of $(X,\tau_X)$. Note that the strong dual of $(Y,\tau_Y)$, that is $(Y',\beta(Y',Y)) = (X'',\beta(X'',X'))$, coincides with $(X,\tau_X)$ by reflexivity. Hence:
	\begin{enumerate}[(i)]
		\item The space $(Y,\tau_Y)$ is a (DF)-space and it is separable because so is its strong dual. Then $(Y,\tau_Y)$ is a quasi-$\ell_{\infty}$-barrelled Hausdorff locally convex topological vector space (check Remark~\ref{Rem:Theorem.conditions}).
		
		\item The strong dual of $(Y,\tau_X)$ coincides with the separable Fr\'echet space $(X,\tau_X)$.
		
		\item The operator $T \in \Lc(X)$ can be seen as the adjoint of the linear map $S:Y\longrightarrow Y$ defined as
		\[
		Su \in Y \text{ such that } \ep{Su}{x} := \ep{u}{Tx} \text{ for every } x \in X.
		\]
		Indeed, the linear map $S$ is the adjoint of $T$, and $T$ coincides with the adjoint of $S$ by reflexivity.
	\end{enumerate}
	By (i), (ii) and (iii) we have that the initial hypothesis of Theorem~\ref{The:lbo->measure} are fulfilled by $(Y,\tau_Y)$, $(X,\tau_X)$ and $T \in\Lc(X)$, so that the conclusion holds. Moreover, note that the corresponding weak topology that appears in the statement $\sigma(X'',X') = \sigma(Y',Y)$ coincides with $\sigma(X,X')$ by reflexivity.
\end{proof}

\subsection{Remarks on Theorem~\ref{The:lbo->measure}}\label{SubSec_2.3:remarks}

Let us include here some examples where Theorem~\ref{The:lbo->measure} can be applied:

\begin{example}[\textbf{Reflexive Fr\'echet spaces}]
	The ``reflexive'' hypothesis is not too restrictive since plenty of interesting Fr\'echet spaces are reflexive. Moreover, the positive part of considering a reflexive space is that we have no restriction on the operators to which we can apply Theorem~\ref{The:lbo->measure}.
	
	Among these spaces we have the important class of {\em Fr\'echet-Montel spaces} like the space of all holomorphic functions $H(\Omega)$ on any open connected subset $\Omega \subset \CC$ equipped with the compact-open topology, also the space of smooth functions $C^{\infty}(\Omega)$ on any open subset $\Omega \subset \RR^n$ equipped with the compact-open topology in all derivatives, or the space of all (real or complex) sequences $\omega=\KK^{\NN}$ endowed with its usual coordinatewise convergence topology.
	
	We refer to \cite[Chapter~11]{Jarchow1981_book} for more on reflexivity.
\end{example}

\begin{example}[\textbf{The case of (DF)-spaces with separable dual}]
	Apart from the reflexive setting, and as pointed out in Remark~\ref{Rem:Theorem.conditions}, the conclusion of Theorem~\ref{The:lbo->measure} also holds when we start considering a (DF)-space $(Y,\tau_Y)$ with separable strong dual and we pick an adjoint operator in that dual space. An important class of spaces fulfilling this property are the {\em (LB)-spaces} with separable dual.
	
	For instance one can consider the so-called {\em K\"othe sequence spaces} (see \cite[Chapter~27]{MeisVogt1997_book}). Indeed, for every {\em K\"othe matrix} $A = (a_{k,j})_{k,j \in \NN}$, as defined in \cite[Page~327]{MeisVogt1997_book} but also explicitly included in this notes (see Example~\ref{Ex:B.lambda} below), the respective {\em K\"othe space} $\lambda^p(A)$ is the strong dual space of an inductive limit of countably many Banach spaces for every $1\leq p\leq \infty$. In particular, when $1<p<\infty$ then $\lambda^p(A)$ is always a reflexive Fr\'echet space (see \cite[Proposition~27.3]{MeisVogt1997_book}) and, even though the spaces $\lambda^1(A)$ and $\lambda^{\infty}(A)$ are not necessarily reflexive (see \cite[Theorem~27.9: Dieudonn\'e-Gomes]{MeisVogt1997_book}), we at least have that $\lambda^1(A)$ is always separable and the strong dual space of an inductive limit of $c_0$-weighted spaces. Thus, Theorem~\ref{The:lbo->measure} applies to every continuous weighted backward shift acting on $\lambda^1(A)$.
\end{example}

In Section~\ref{Sec_4:lbo} we give explicit examples of locally bounded orbits for some classical operators acting on the already mentioned spaces $H(\CC)$, $\omega$ and $\lambda^p(A)$, but let us end this part of the paper with a possible generalization of Theorem~\ref{The:lbo->measure} for Polish dynamical systems:

\begin{remark}
	If one reads carefully the proof of Theorem~\ref{The:lbo->measure} it is obvious that the construction holds because we assume the existence of a reiteratively recurrent point admitting a basis of neighbourhoods whose intersection with the respective orbit is compact for the weak topology. We could give a similar definition in the general Polish dynamical systems setting:
	
	\begin{definition}
		Let $T:X\longrightarrow X$ be a continuous map acting on a Polish space $(X,\tau_X)$ and let $\tau$ be any Hausdorff topology on $X$. We say that a point $x \in X$ has a {\em locally $\tau$-compact orbit for $T$} if any of the following equivalent conditions holds:
		\begin{enumerate}[(i)]
			\item for each $\tau_X$-neighbourhood $U$ of $x$ there is a second $\tau_X$-neighbourhood $V$ of $x$ such that $\cl{V \cap \orb_T(x)}^{\tau}$ is a $\tau$-compact set contained in $U$;
			
			\item there exists a decreasing basis of $\tau_X$-neighbourhoods $(U_n)_{n\in\NN}$ of $x$ fulfilling that $\cl{U_{n+1} \cap \orb_T(x)}^{\tau}$ is a $\tau$-compact set contained in $U_n$ for every $n \in \NN$.
		\end{enumerate}
	\end{definition}
	
	Note that, when $(Y,\tau_Y)$ is a quasi-$\ell_{\infty}$-barrelled Hausdorff locally convex topological vector space whose strong dual $(X,\tau_X):=(Y',\beta(Y',Y))$ is a Fr\'echet space as in Theorem~\ref{The:lbo->measure}, then a vector $x \in X$ has a locally $\sigma(X,Y)$-compact orbit for an operator $T \in \Lc(X)$ if and only if $x \in \lbo(T)$. A completely similar proof to that of Theorem~\ref{The:lbo->measure} shows now the following extension of \cite[Theorem~2.3]{GriLo2023}:
	
	\begin{theorem}
		Let $(X,T)$ be a Polish dynamical system. Assume that $X$ is endowed with a Hausdorff topology $\tau$ which fulfills {\em($\alpha$)}, {\em($\beta$)}, {\em($\gamma$)} and {\em ($\delta$)} with respect to the original Polish topology $\tau_X$ on $X$. Given a point $x_0 \in \RRec(T)$ with a locally $\tau$-compact orbit for $T$ one can find a $T$-invariant probability measure $\mu_{x_0}$ on $(X,\Bi(X))$ such that
		\[
		x_0 \in \supp(\mu_{x_0}) \subset \cl{\orb_T(x_0)}^{\tau}.
		\]
		Moreover, if the set of reiteratively recurrent points with locally $\tau$-compact orbit for $T$ is dense in $X$, then there exits a $T$-invariant probability measure $\mu$ on $(X,\Bi(X))$ with full support.
	\end{theorem}
\end{remark}

\section{Applications, optimality, almost-$\Fc$-recurrence and chaos}\label{Sec_3:applications}

In this section we apply Theorem~\ref{The:lbo->measure} to certain linear dynamical systems $T \in \Lc(X)$ with many locally bounded orbits, obtaining for them a very strong equivalence between two generally distinguished recurrence notions. We also discuss the optimality of Theorem~\ref{The:lbo->measure} in terms of Banach limits and Furstenberg families, which implies the optimality of the announced equivalence. We finally use again the concept of {\em locally bounded orbit} to extend, from the Banach to the Fr\'echet setting, two results from the recent works \cite{CarMur2022_IEOT,CarMur2022_arXiv} regarding the notions of {\em almost-$\Fc$-recurrence} and {\em Devaney chaos}.

\subsection{Applications and optimality of the measures constructed}\label{SubSec_3.1:applications+optimality}

By Theorem~\ref{The:lbo->measure} we can construct invariant measures from every single reiteratively recurrent vector and then we can show the existence of many frequently recurrent points, which have a much stronger recurrent behaviour than that of reiterative recurrence (see \cite[Lemma~3.1]{GriLo2023}). These arguments were already exhibited in \cite[Theorem~1.3]{GriLo2023} for adjoint operators acting on dual Banach spaces but now we can obtain an extension of such a result in our ``dual Fr\'echet setting'':

\begin{proposition}\label{Pro:application}
	Let $(Y,\tau_Y)$ be a quasi-$\ell_{\infty}$-barrelled Hausdorff locally convex topological vector space, assume that its strong dual $(X,\tau_X):=(Y',\beta(Y',Y))$ is a separable Fr\'echet space, and let $T \in \Lc(X)$ be the adjoint of some linear map $S:Y\longrightarrow Y$. Then we have the inclusions
	\[
	\RRec(T) \cap \lbo(T) \subset \bigcup \{ \supp(\mu) \ ; \ \mu \textup{ is a $T$-invariant probability measure on } \Bi(X) \} \subset \cl{\FRec(T)}.
	\]
	In particular, this holds for every $T \in \Lc(X)$ as soon as $(X,\tau_X)$ is a separable reflexive Fr\'echet space.
\end{proposition}
\begin{proof}
	Using Theorem~\ref{The:lbo->measure} we have that for any vector $x_0 \in \RRec(T) \cap \lbo(T)$ there is a $T$-invariant probability measure $\mu_{x_0}$ on $(X,\Bi(X))$ such that $x_0 \in \supp(\mu_{x_0})$. Using now \cite[Lemma~3.1]{GriLo2023} we get that $\supp(\mu_{x_0}) \subset \cl{\FRec(T)}$ and the result follows.
\end{proof}

If we guarantee that every reiteratively recurrent vector has a locally bounded orbit, then we recover the main results from \cite[Section~3]{GriLo2023} but in our more general ``dual Fr\'echet setting'':

\begin{corollary}\label{Cor:application}
	Let $(Y,\tau_Y)$ be a quasi-$\ell_{\infty}$-barrelled Hausdorff locally convex topological vector space, assume that its strong dual $(X,\tau_X):=(Y',\beta(Y',Y))$ is a separable Fr\'echet space, and let $T \in \Lc(X)$ be the adjoint of some linear map $S:Y\longrightarrow Y$. If we suppose that $\RRec(T) \subset \lbo(T)$ then
	\[
	\cl{\FRec(T)} = \cl{\RRec(T)}
	\]
	and, moreover, the next statements are equivalent:
	\begin{enumerate}[{\em(i)}]
		\item $T$ admits an invariant probability measure with full support;
		
		\item $T$ is frequently recurrent (that is, the set $\FRec(T)$ is dense in $X$);
		
		\item $T$ is reiteratively recurrent (that is, the set $\RRec(T)$ is dense in $X$).
	\end{enumerate}
	In particular, the inclusion $\RRec(T) \subset \lbo(T)$ holds if we assume any of the following conditions:
	\begin{enumerate}[--]
		\item the space $(X,\tau_X)$ is Banach (that is, $(X,\tau_X)$ is a locally bounded Fr\'echet space);
		
		\item or the operator $T \in \Lc(X)$ is power-bounded (that is, every $T$-orbit is bounded).
	\end{enumerate}
\end{corollary}
\begin{proof}
	By definition we always have the inclusion $\FRec(T) \subset \RRec(T)$ for every operator $T \in \Lc(X)$, and even for every Polish dynamical system $(X,T)$. The converse inclusion follows from Proposition~\ref{Pro:application} since we are assuming that $\RRec(T) \subset \lbo(T)$. Moreover, (i) $\Rightarrow$ (ii) follows from \cite[Lemma~3.1]{GriLo2023}, we have the equivalence (ii) $\Leftrightarrow$ (iii) since $\cl{\FRec(T)}=\cl{\RRec(T)}$, and (iii) $\Rightarrow$ (i) follows from Theorem~\ref{The:lbo->measure}.
	
	Note that if $(X,\tau_X)$ is a Banach space, or if $T \in \Lc(X)$ is a power-bounded operator, then we clearly have the inclusions $\RRec(T) \subset X \subset \lbo(T)$.
\end{proof}

We consider worth mentioning again that Corollary~\ref{Cor:application} contains \cite[Theorem~1.3]{GriLo2023}, which is the original Banach version of the result. Moreover, following the arguments employed in \cite{GriLo2023} one can prove extended versions of Proposition~\ref{Pro:application} and Corollary~\ref{Cor:application} for ``product'' and ``inverse'' linear dynamical systems. This was deeply studied in \cite[Sections~5 and 6]{GriLo2023} and we will not develop it further here.

Let us now focus on the optimality of the measures obtained in Section~\ref{Sec_2:invariant.measures}. First of all we have to mention that Theorem~\ref{The:lbo->measure} does not hold, and hence Proposition~\ref{Pro:application} and Corollary~\ref{Cor:application} are no longer true, outside the ``dual/reflexive setting'' described in Section~\ref{Sec_2:invariant.measures}. Indeed, in \cite[Section~5]{BoGrLoPe2022} there are explicit examples of linear dynamical systems, acting on non-dual spaces, that have plenty of reiteratively recurrent vectors (they have a co-meager and hence dense set of such vectors) but no non-zero frequently recurrent vector, so that the only invariant probability measure that these operators admit is the trivial Dirac delta $\delta_0$ (see \cite[Theorem~5.7 and Corollary~5.8]{BoGrLoPe2022}).

A second question regarding the optimality of Theorem~\ref{The:lbo->measure} is whether we can weaken or not the ``reiterative recurrence'' assumption. Indeed, if we could construct invariant measures from vectors presenting a weaker recurrent behaviour than that of reiterative recurrence, then Proposition~\ref{Pro:application} and also Corollary~\ref{Cor:application} would show the existence of frequently recurrent vectors but starting from a condition weaker than reiterative recurrence. We are about to show that we cannot find such a weaker property, but in order to give a complete answer to this question let us recall the following definitions already used and deeply studied in the works \cite{BesMePePu2016,BesMePePu2019,BoGr2018,BoGrLoPe2022,CarMur2022_IEOT,CarMur2022_MS,CarMur2022_arXiv,GriLo2023,GriLoPe2022,LoMe2023}:

\begin{definition}\label{Def:F-rec}
	If we denote by $\Part(\NN)$ the {\em power set} of the set of positive integers $\NN$, then:
	\begin{enumerate}[(a)]
		\item a collection of sets $\Fc \subset \Part(\NN)$ is called a {\em Furstenberg family} (or just a {\em family} for short) if $\varnothing \notin \Fc$ and for every $A \in \Fc$ the inclusion $A \subset B \subset \NN$ implies that $B\in \Fc$;
		
		\item and given an operator $T \in \Lc(X)$ and a Furstenberg family $\Fc \subset \Part(\NN)$, a vector $x\in X$ is called {\em $\Fc$-recurrent for $T$} if for every neighbourhood $U$ of $x$ the return set $\Nc_T(x,U) = \{ n \in \NN \ ; \ T^nx \in U \}$ belongs to $\Fc$. We will denote by $\Fc\Rec(T)$ the {\em set of $\Fc$-recurrent vectors for $T$}, and we will say that the operator $T$ is {\em $\Fc$-recurrent} whenever the set $\Fc\Rec(T)$ is dense in $X$.
	\end{enumerate}
\end{definition}

Note that {\em frequent} and {\em reiterative recurrence}, as defined in the Introduction of this paper, are two particular cases of {\em $\Fc$-recurrence} that appear precisely when $\Fc$ is chosen to be:
\begin{enumerate}[--]
	\item the {\em family of sets with positive lower density}, which will be denoted as in \cite{BesMePePu2019,CarMur2022_arXiv,GriLo2023} by
	\[
	\Dinf := \left\{ A \subset \NN \ ; \ \dinf(A) = \liminf_{N\to\infty} \frac{\# (A \cap [1,N])}{N} > 0 \right\};
	\]
	
	\item or the {\em family of sets with positive upper Banach density}, which will be denoted as in \cite{BesMePePu2019,CarMur2022_arXiv,GriLo2023} by
	\[
	\BDsup := \left\{ A \subset \NN \ ; \ \Bdsup(A) = \limsup_{N\to\infty} \left( \max_{m \geq 0} \frac{\# (A \cap [m+1,m+N])}{N} \right) > 0 \right\}.
	\]
\end{enumerate}

Therefore, our search for a property weaker than ``reiterative recurrence'', yet still enabling us to derive the conclusion of Theorem~\ref{The:lbo->measure}, can be formulated (and was implicitly asked by A. Avil\'es to the author of this paper in the context of the original Banach space result \cite[Theorem~1.3]{GriLo2023}) as follows:
\begin{enumerate}[--]
	\item {\em Is there any Furstenberg family $\Fc \subset \Part(\NN)$ fulfilling that $\BDsup \subsetneq \Fc$ and such that the conclusion of Theorem~\ref{The:lbo->measure} still holds for every vector $x_0 \in \Fc\Rec(T) \cap \lbo(T)$?}
\end{enumerate}
Recall that, in \textbf{Claim~\ref{Claim_1}} of Theorem~\ref{The:lbo->measure}, it is necessary that given a set $A \in \BDsup$ then one can find a Banach limit $\mf_A:\ell^{\infty}\longrightarrow \RR$ such that $\mf_A(\1_A)=\Bdsup(A)>0$, since this Banach limit is crucial to construct the strictly positive invariant measure required. Thus, the optimal form of Theorem~\ref{The:lbo->measure} in terms of Furstenberg families would appear if we replace $\BDsup$ by the ``apparently new'' family
\[
\BL := \left\{ A \subset \NN \ ; \ \text{there exists a Banach limit } \mf_A : \ell^{\infty} \longrightarrow \RR \text{ with } \mf_A(\1_A)>0 \right\}.
\]

\begin{proposition}\label{Pro:BDsup=BL}
	We have the following equality of Furstenberg families $\BDsup = \BL$.
\end{proposition}
\begin{proof}
	The inclusion $\BDsup \subset \BL$ was already discussed in \cite[Fact~2.3.1]{GriLo2023} and follows since given $A \in \BDsup$ we can find a strictly increasing sequence of positive integers $(N_k)_{k\in\NN}$ and a sequence of intervals of positive integers $(J_k = \{ j_k+1, j_k+2, ..., j_k+N_k\})_{k\in\NN}$ such that
	\[
	\lim_{k\to\infty} \frac{\# (A \cap J_k)}{N_k} = \Bdsup(A) > 0,
	\]
	thus $\mf_A:\ell^{\infty}\longrightarrow \RR$ with $\mf_A(\phi) := \lim_{\Uc} \frac{1}{N_k} \sum_{n\in J_k} \phi_n$ for each $\phi=(\phi_n)_{n\in\NN} \in \ell^{\infty}$, where $\Uc \subset \Part(\NN)$ is a fixed non-principal ultrafilter on $\NN$, is a Banach limit for which $\mf_A(\1_A) = \Bdsup(A)>0$.
	
	Conversely, given a set $A \in \BL$ there is a Banach limit $\mf_A:\ell^{\infty}\longrightarrow \RR$ such that $\mf_A(A)>0$. Using now \cite[Theorem~1]{Sucheston1964}, which asserts that the maximum value that a Banach limit can get on a sequence $\phi=(\phi_n)_{n\in\NN} \in \ell^{\infty}$ is precisely the value given by the functional $M:\ell^{\infty}\longrightarrow \RR$ with
	\[
	M(\phi) := \lim_{N\to\infty} \left( \sup_{m \geq 0} \frac{1}{N} \sum_{j=m+1}^{m+N} \phi_j \right),
	\]
	we clearly have that $\Bdsup(A) = M(\1_A) \geq \mf_A(\1_A) > 0$ and hence $A \in \BDsup$.
\end{proof}

Proposition~\ref{Pro:BDsup=BL} shows that the measures from Theorem~\ref{The:lbo->measure} (but also those constructed in \cite{GriLo2023}) are optimal in terms of Banach limits and Furstenberg families. This observation slightly improves the classical result of Oxtoby and Ulam \cite[Theorem~1]{OxUlam1939} for Polish dynamical systems:

\begin{proposition}
	Let $(X,T)$ be a Polish dynamical system acting on the Polish space $(X,\tau_X)$. The following statements are equivalent and optimal in terms of Banach limits and Furstenberg families:
	\begin{enumerate}[{\em(i)}]
		\item there exists a positive finite $T$-invariant Borel measure $\mu$ on $(X,\Bi(X))$;
		
		\item there exist a point $x \in X$ and a $\tau_X$-compact set $K \subset X$ such that $\displaystyle \lim_{N\to\infty} \frac{\# (\Nc_T(x,K) \cap [1,N])}{N} > 0$;
		
		\item there exist a Hausdorff topology $\tau$ on $X$ fulfilling the properties {\em($\alpha$)}, {\em($\beta$)}, {\em($\gamma$)} and {\em($\delta$)} with respect to $\tau_X$, a point $x \in X$ and a $\tau$-compact set $K \subset X$ such that $\Bdsup(\Nc_T(x,K))>0$.
	\end{enumerate}
\end{proposition}
\begin{proof}
	(i) $\Rightarrow$ (ii) was already shown in \cite[Theorem~1]{OxUlam1939} by using similar arguments to those employed in \cite[Lemma~3.1]{GriLo2023}. We have (ii) $\Rightarrow$ (iii) since $\tau_X$ fulfills the properties ($\alpha$), ($\beta$), ($\gamma$) and ($\delta$) with respect to itself. Finally, (iii) $\Rightarrow$ (i) follows from \cite[Lemma~2.1]{GriLo2023}. The optimality in terms of Banach limits, Furstenberg families and densities follows from Proposition~\ref{Pro:BDsup=BL} and \cite[Lemma~2.1]{GriLo2023}.
\end{proof}

\subsection{Almost-$\Fc$-recurrence and Devaney chaos on dual Fr\'echet spaces}\label{SubSec_3.2:PF+chaos}

We finish Section~\ref{Sec_3:applications} by showing two more {\em locally bounded orbit}'s applications, which are not related to the existence of invariant measures but that also use the ``dual Fr\'echet setting'' from Section~\ref{Sec_2:invariant.measures}. Both applications are based on recent investigations from Rodrigo Cardeccia and Santiago Muro, who have successfully used the ``adjoint operators acting on dual Banach spaces setting'' to obtain strong results regarding the notions of {\em almost-$\Fc$-recurrence} and {\em Devaney chaos} (see \cite{CarMur2022_arXiv} and \cite{CarMur2022_IEOT} respectively). 

Let us start by the notion of {\em almost-$\Fc$-recurrence}, name recently coined by Cardeccia and Muro for general Furstenberg families in \cite[Definition~3.2]{CarMur2022_arXiv}, although this concept was previously considered for particular families (under very different names) by several authors such as Costakis and Parissis~\cite{CoPa2012}, Badea and Grivaux~\cite[Proposition~4.6]{BadGri2007} and Grivaux and Matheron~\cite[Section~2.5]{GriMa2014}:

\begin{definition}[\textbf{\cite[Definition~3.2]{CarMur2022_arXiv}}]\label{Def:al-F-rec}
	Given a Furstenberg family $\Fc \subset \Part(\NN)$ we say that an operator $T \in \Lc(X)$ is {\em almost-$\Fc$-recurrent} if for every non-empty open subset $U$ of $X$ there exists a vector $x_U \in U$ such that the return set $\Nc_T(x_U,U) = \{ n \in \NN \ ; \ T^nx_U \in U \}$ belongs to $\Fc$.
\end{definition}

We remark that the notion of {\em almost-$\Fc$-recurrence} is highly inspired in the so-called {\em $\Pc_{\Fc}$ property} introduced in 2018 by Puig \cite{Puig2018}: an operator $T \in \Lc(X)$ has the {\em $\Pc_{\Fc}$ property} if for every non-empty open subset $U$ of $X$ there exists a vector $x_U \in X$ such that the return set $\Nc_T(x_U,U)$ belongs to $\Fc$. The only difference between these two concepts is the relation ``$x_U \in U$'', so that almost-$\Fc$-recurrence is slightly stronger than the $\Pc_{\Fc}$ property. However, both concepts coincide whenever $\Fc \subset \Part(\NN)$ is {\em left-invariant}, that is, if for every $A \in \Fc$ and $k \in \NN$ the set $(A-k) \cap \NN$ belongs to $\Fc$, where $A-k := \{ a-k \ ; \ a \in A \}$. Indeed, note that if $\Fc$ is left-invariant and there exists some $x_0 \in X\setminus U$ fulfilling that $\Nc_T(x_0,U) \in \Fc$ and for which $n_0 := \min(\Nc_T(x_0,U))$, then the vector $x_U:=T^{n_0}x \in U$ clearly fulfills that $(\Nc_T(x_0,U)-n_0) \cap \NN = \Nc_T(x_U,U)$ and hence $\Nc_T(x_U,U) \in \Fc$. It is worth mentioning that the usual families considered in the literature (such as $\BDsup$ and $\Dinf$ mentioned in Subsection~\ref{SubSec_3.1:applications+optimality}) are left-invariant (see also \cite{BoGr2018}, \cite[Section~3]{CarMur2022_arXiv} or \cite[Example~4.2]{GriLoPe2022}), so it is natural to just focus on the similarities/differences between {\em almost-$\Fc$-recurrence} and the standard notion of {\em $\Fc$-recurrence}.

As observed in \cite[Section~3]{CarMur2022_arXiv}, since the definition of $\Fc$-recurrence (see Definition~\ref{Def:F-rec}) requires the density of the set $\Fc\Rec(T)$, and since the $\Fc$-recurrent vectors return to each of their neighbourhoods with ``frequency $\Fc$'', it follows from Definitions~\ref{Def:F-rec} and \ref{Def:al-F-rec} that the concept of {\em almost-$\Fc$-recurrence} is (at least formally) weaker than that of {\em $\Fc$-recurrence}. Indeed, it is asked in~\cite[Section~5]{CarMur2022_arXiv} and still open for the moment, if both properties coincide or not for continuous linear operators. This question encourages to search for results similar in spirit to \cite[Theorem~1.3]{GriLo2023} and Corollary~\ref{Cor:application}, where several notions of $\Fc$-recurrence are shown to coincide for different Furstenberg families. In fact, one of the main lines of though in the recent work from Cardeccia and Muro \cite{CarMur2022_arXiv} is to search for families $\Fc \neq \Gc \subset \Part(\NN)$ fulfilling that almost-$\Fc$-recurrence and almost-$\Gc$-recurrence are equivalent properties. This led to the so-called {\em block families} (see \cite[Definition~3.3]{CarMur2022_arXiv} but also \cite{Glasner2004,HuangLiYe2012,Li2011}):

\begin{definition}
	For a Furstenberg family $\Fc \subset \Part(\NN)$ we define $b\Fc$, the associated {\em block family}, in the following way: a set $B \subset \NN$ belongs to $b\Fc$ if there exists some $A_B \in \Fc$ such that for each finite subset $F \subset A_B$ there is some $n_F \in \NN \cup \{0\}$ for which $F + n_F := \{ f+n_F \ ; \ f \in F \} \subset B$.
\end{definition}

Roughly speaking, the block family $b\Fc$ obtained from a given $\Fc$ is the collection of sets that contain every finite block from a fixed set of the original family, but possibly translated. Some general basic properties (such as the inclusion $\Fc \subset b\Fc$) and examples (such as the equality $b\Dinf = \BDsup$) are exposed in \cite[Section~3]{CarMur2022_arXiv}, and the authors prove in \cite[Theorem~3.12]{CarMur2022_arXiv} the equivalence between the notion of almost-$\Fc$-recurrence and that of almost-$b\Fc$-recurrence, for adjoint operators acting on dual Banach spaces and every left-invariant Furstenberg family $\Fc \subset \Part(\NN)$. Using {\em locally bounded orbits} and the ``dual Fr\'echet setting'' exposed at Section~\ref{Sec_2:invariant.measures} we can obtain the following extension of such a result:

\begin{theorem}\label{The:lbo->PF.property}
	Let $(Y,\tau_Y)$ be a quasi-$\ell_{\infty}$-barrelled Hausdorff locally convex topological vector space, assume that its strong dual $(X,\tau_X):=(Y',\beta(Y',Y))$ is a separable Fr\'echet space, and let $T \in \Lc(X)$ be the adjoint of some linear map $S:Y\longrightarrow Y$. If the set $[b\Fc]\Rec(T) \cap \lbo(T)$ is dense in $(X,\tau_X)$ for a left-invariant Furstenberg family $\Fc \subset \Part(\NN)$, then $T$ is almost-$\Fc$-recurrent.
\end{theorem}
\begin{proof}
	Let $U$ be an arbitrary but fixed non-empty $\tau_X$-open subset of $X$. By assumption there is some $x_0 \in U \cap [b\Fc]\Rec(T) \cap \lbo(T)$. Since $x_0$ has a locally bounded orbit we can find a $\sigma(X,Y)$-closed $\tau_X$-neighbourhood $V$ of $x_0$ fulfilling that $V \cap \orb_T(x_0)$ is a countable $\tau_X$-bounded set and hence an equicontinuous set in $Y'$ by the quasi-$\ell_{\infty}$-barrelled assumption on $(Y,\tau_Y)$. Without lost of generality we can assume that $V \subset U$, and by the Alaoglu-Bourbaki theorem (see \cite[Section~8.5]{Jarchow1981_book}) we have that
	\[
	K := \cl{V \cap \orb_T(x_0)}^{\sigma(X,Y)}
	\]
	is a $\sigma(X,Y)$-compact set fulfilling that $K \subset V \subset U$. Note also that we have the inclusions
	\[
	\Nc_T(x_0,V) = \Nc_T(x_0,V \cap \orb_T(x_0)) \subset \Nc_T(x_0,K) \subset \Nc_T(x_0,U),
	\]
	which imply that $\Nc_T(x_0,K) \in b\Fc$ by the hereditarily upward property of the Furstenberg family $\Fc$. By definition of block family there exists a set $A \in \Fc$ such that for every finite subset $F \subset A$ there is some $n_F \in \NN \cup \{0\}$ fulfilling that $F + n_F := \{ f+n_F \ ; \ f \in F \} \subset \Nc_T(x_0,K)$. Set $r_0:=\min(A)$.
	
	Following now the proof of \cite[Theorem~3.12]{CarMur2022_arXiv}, for each $n \in \NN$ let $A_n := A \cap [1,n]$ and pick some $a_n \in \NN \cup \{0\}$ such that $T^{a_n+r}(x_0) \in K$ for every $r \in A_n$. Since $K$ is $\sigma(X,Y)$-compact and $(Y,\tau_Y)$ is separable, property ($\gamma$) from Lemma~\ref{Lem:top.conditions} shows that $K$ is also $\sigma(X,Y)$-metrizable. Thus, there exist a subsequence $(a_{n_k})_{k\in\NN}$ and a vector $x_U \in K$ satisfying that $(T^{a_{n_k}+r_0}(x_0))_{k\in\NN}$ is $\sigma(X,Y)$-convergent to $x_U \in K \subset U$. We finally claim that $(A-r_0) \cap \NN \subset \Nc_T(x_U,U)$ and hence that $\Nc_T(x_U,U) \in \Fc$ by left-invariance, which finishes the proof since $U$ was arbitrary. Indeed, for any $r \in A$ with $r>r_0$,
	\[
	T^{r-r_0}(x_U) = T^{r-r_0}\left( \sigma(X,Y)\text{ --}\lim_{k\to\infty} T^{a_{n_k}+r_0}(x_0) \right) = \sigma(X,Y)\text{ --}\lim_{k\to\infty} T^{a_{n_k}+r}(x_0),
	\]
	and $T^{a_{n_k}+r}(x_0) \in K$ provided that $r \in A_{n_k}$, that is, as soon as $n_k>r$. The $\sigma(X,Y)$-compactness of the set $K$ implies that $T^{r-r_0}(x_U) \in K \subset U$ and hence $r-r_0 \in \Nc_T(x_U,U)$ as we had to show.
\end{proof}

We refer the reader to \cite[Section~3]{CarMur2022_arXiv} for more on {\em almost-$\Fc$-recurrence} and the possible consequences of Theorem~\ref{The:lbo->PF.property}. Let us now focus on the notion of chaos: a linear dynamical system $T \in \Lc(X)$ is called {\em Devaney chaotic} (or just {\em chaotic} for short) if $T$ is hypercyclic and the set of $T$-periodic vectors is dense. In \cite[Theorem~3.11]{CarMur2022_IEOT} Cardeccia and Muro characterize the notion of chaos, for adjoint operators acting on dual Banach spaces, as a concrete case of $\Fc$-hypercyclicity: they introduce the family $\AP_b \subset \Part(\NN)$ of {\em sets containing arbitrarily long arithmetic progressions with a fixed common bounded difference}, and they show that an adjoint operator acting on a dual Banach space $T \in \Lc(X)$ is chaotic if and only if $T$ is {\em $\AP_b$-hypercyclic}, but also if and only if the operator $T$ is hypercyclic and has {\em dense small periodic sets}. We can again extend this result to our ``dual Fr\'echet setting'':

\begin{theorem}\label{The:lbo->chaos}
	Let $(Y,\tau_Y)$ be a quasi-$\ell_{\infty}$-barrelled Hausdorff locally convex topological vector space, assume that its strong dual $(X,\tau_X):=(Y',\beta(Y',Y))$ is a separable Fr\'echet space, and let $T \in \Lc(X)$ be the adjoint of some linear map $S:Y\longrightarrow Y$. The following assertions are equivalent:
	\begin{enumerate}[{\em(i)}]
		\item $T$ is hypercyclic and the set $[\AP_b]\Rec(T) \cap \lbo(T)$ is dense in $X$;
		
		\item $T$ is hypercyclic and has dense small bounded periodic sets;
		
		\item $T$ is (Devaney) chaotic. 
	\end{enumerate}
\end{theorem}\newpage
The interested reader can find the precise definition of the previous concepts in \cite{CarMur2022_IEOT}, but let us just mention that $T \in \Lc(X)$ is said to have {\em dense small bounded periodic sets} if every non-empty open subset $U \subset X$ contains a bounded set $Y \subset U$ such that $T^p(Y) \subset Y$ for some $p \in \NN$. The proof is analogous to that of \cite[Theorem~3.11]{CarMur2022_IEOT}, but using the arguments of Theorems~\ref{The:lbo->measure} and \ref{The:lbo->PF.property} regarding locally bounded orbits, and we just include a sketch of the proof:
\begin{proof}[\textit{Sketch to prove Theorem~\ref{The:lbo->chaos}}]
	For (iii) $\Rightarrow$ (i) recall that every periodic vector is trivially $\AP_b$-recurrent; to prove (i) $\Rightarrow$ (ii) one can adapt the arguments employed in both Theorems~\ref{The:lbo->measure} and \ref{The:lbo->PF.property} regarding locally bounded orbits, but following the proof of \cite[Proposition~3.9 and Lemma~3.10]{CarMur2022_IEOT}; and the final implication (ii) $\Rightarrow$ (iii) follows as in \cite[Theorem~3.11]{CarMur2022_IEOT} because every bounded set in $(X,\tau_X)$ is relatively $\sigma(X,Y)$-compact (recall that $(Y,\tau_Y)$ is quasi-barrelled by separability, see Remark~\ref{Rem:Theorem.conditions}).
\end{proof}

Note that Theorems~\ref{The:lbo->measure}, \ref{The:lbo->PF.property} and \ref{The:lbo->chaos} are just modest extensions to the Fr\'echet-space setting of three results originally proved for operators acting on Banach spaces. However, and as we are about to show in the following section, the main classical examples of operators in Linear Dynamics present plenty of locally bounded orbits so that the results obtained here will usually be applicable.

\section{More on locally bounded orbits}\label{Sec_4:lbo}

In this final section we elaborate further on {\em locally bounded orbits} (see Definition~\ref{Def:lbo}). In particular, we study the topological and dynamical structure of the set $\lbo(T)$, we give some explicit examples of locally and non-locally bounded orbits for operators acting on (non-Banach) Fr\'echet spaces, and we argue which kind of strong recurrence is needed for a locally bounded orbit to be bounded. We also include a final subsection with some open problems regarding the main results of this paper.

\subsection{Stability results and explicit examples}

When studying a set of vectors with some dynamical property with respect to an operator $T \in \Lc(X)$, in our case the set $\lbo(T)$, many properties may be observed. First of all we can look at the size of such a set: we have already argued in Remark~\ref{Rem:lbo.1} that $\lbo(T)$ can be the whole space, it could also be a dense but meager set, or even the singleton set formed by the zero-vector $\lbo(T)=\{0\}$ (see \cite{Golinski2013}). Let us show that, at least, the set $\lbo(T)$ always contains the linear span of the $T$-eigenvectors:


\begin{proposition}\label{Pro:eigenvectors->lbo}
	For every $T \in \Lc(X)$ acting on (real or complex) Fr\'echet space $X$ we have the inclusion $\lspan(\operatorname{Eig}(T)) \subset \lbo(T)$, where $\operatorname{Eig}(T) := \{ x \in X \ ; \ Tx = \lambda x \text{ for some } \lambda \in \KK \}$.
\end{proposition}
\begin{proof}
	Recall that $X \setminus \Rec(T) \subset \lbo(T)$ as argued in Remark~\ref{Rem:lbo.1}, so that we just have to check the inclusion $\lspan(\operatorname{Eig}(T)) \cap \Rec(T) \subset \lbo(T)$. We start by noticing that for each $x \in \lspan(\operatorname{Eig}(T))$ there exist non-zero numbers $\alpha_1,...,\alpha_k \in \KK\setminus\{0\}$, eigenvalues $\lambda_1,...,\lambda_k \in \KK$ and vectors $x_1,...,x_k \in \operatorname{Eig}(T)$ such that $x = \sum_{1\leq j\leq k} \alpha_j x_j$, where $Tx_j=\lambda_j x_j$ and $\lambda_j \neq \lambda_l$ for all $1\leq j\neq l\leq k$. Hence, the set of eigenvectors $\{x_1,...,x_k\} \subset X$ is linearly independent and, since $T^nx - x = \sum_{1\leq j\leq k} (\lambda_j^n - 1) \cdot \alpha_j x_j$ for each $n \in \NN$, it follows that $x \in \Rec(T)$ if and only if $|\lambda_j| = 1$ for all $1\leq j \leq k$. We deduce that, if $x \in \lspan(\operatorname{Eig}(T)) \cap \Rec(T)$, then the orbit $\orb_T(x)$ is a compact set and finally $x \in \lbo(T)$.
\end{proof}

Note that the linear subspace $\lspan(\operatorname{Eig}(T)) \subset X$ is a $T$-invariant set for every operator $T \in \Lc(X)$, so that a natural property to look at is the $T$-invariance of $\lbo(T)$:

\begin{proposition}\label{Pro:invertible->invariant}
	If a continuous linear operator $T \in \Lc(X)$ is invertible, then $T(\lbo(T)) \subset \lbo(T)$.
\end{proposition}
\begin{proof}
	Given any $x_0 \in \lbo(T)$ set $x_1 := T(x_0)$. By definition there exists a neighbourhood $U_0$ of $x_0$ such that $U_0 \cap \orb_T(x_0)$ is a bounded set, and by the continuity of $T^{-1}$ we can find a neighbourhood $U_1$ of $x_1$ such that $T^{-1}(U_1) \subset U_0$. We claim that $U_1 \cap \orb_T(x_1) \subset T(U_0 \cap \orb_T(x_0))$, which would finish the proof since the image of a bounded set by a continuous operator is again bounded. Indeed, given any vector $x \in U_1 \cap \orb_T(x_1)$ there exists some $n \in \NN$ such that $x = T^nx_1 = T^{n+1}x_0 \in U_1$, but then the vector $y := T^{-1}(x) = T^{n-1}x_1 = T^nx_0$ fulfills that $y \in U_0 \cap \orb_T(x_0)$ and $x=Ty$.
\end{proof}


Let us apply Propositions~\ref{Pro:eigenvectors->lbo} and \ref{Pro:invertible->invariant} to the following well-known class of operators:

\begin{example}[\textbf{Birkhoff, MacLane and differential operators}]
	For a complex number $a \in \CC\setminus\{0\}$ consider the {\em translation operator}, also called the {\em Birkhoff operator}, $T_a : H(\CC) \longrightarrow H(\CC)$ on the space of entire functions endowed with the usual compact-open topology, where $[T_a(f)](z) := f(z+a)$ for each $f \in H(\CC)$. The operator $T_a$ is invertible and chaotic (see \cite[Example~2.35]{GrPe2011_book}), so that the set $\lbo(T_a)$ is $T_a$-invariant by Proposition~\ref{Pro:invertible->invariant}, and a dense but meager set in $H(\CC)$. A function $f \in H(\CC)$ belongs to $\lbo(T_a)$ if and only if there exist $k_0 \in \NN$, $\eps>0$ and $w=(w_j)_{j\in\NN} \in \ ]0,+\infty[^{\NN}$ such that
	\[
	\forall n \in \NN \quad \text{with} \quad \max_{|z| \leq k_0} \left|f(z) - f(z+na)\right| < \eps, \quad \text{then} \quad \max_{|z| \leq j} \left|f(z+na)\right| \leq w_j \quad \forall j \in \NN.
	\]
	
	In the space of entire functions we can also consider the {\em standard differential operator}, commonly known as the {\em MacLane operator}, $D : H(\CC) \longrightarrow H(\CC)$, where $[D(f)](z):=f'(z)$ for each $f \in H(\CC)$. The operator $D$ is not invertible but chaotic (see again \cite[Example~2.35]{GrPe2011_book}), so that Proposition~\ref{Pro:invertible->invariant} does not apply although the set $\lbo(D)$ is again a dense but meager set in $H(\CC)$. A function $f \in H(\CC)$ belongs to $\lbo(D)$ if and only if there exist $k_0 \in \NN$, $\eps>0$ and $w=(w_j)_{j\in\NN} \in \ ]0,+\infty[^{\NN}$ such that
	\[
	\forall n \in \NN \quad \text{with} \quad \max_{|z| \leq k_0} \left|f(z) - f^{(n)}(z)\right| < \eps, \quad \text{then} \quad \max_{|z| \leq j} \left|f^{(n)}(z)\right| \leq w_j \quad \forall j \in \NN,
	\]
	where we are using the notation $f^{(n)}(z) := [D^n(f)](z)$.
	
	Birkhoff and MacLane operators are a particular case of the so-called {\em differential operators}, usually denoted by $\varphi(D) : H(\CC) \longrightarrow H(\CC)$, where $\varphi \in H(\CC)$ is an {\em entire function of exponential type} acting on the {\em standard differential operator} $D : H(\CC) \longrightarrow H(\CC)$. These operators were originally studied by Godefroy and Shapiro (see \cite[Section~5]{GodSha1991}) and it is well-known that for each $a \in \CC\setminus\{0\}$ we have the equality $T_a = e^{aD}$ but also $D = p(D)$ for the polynomial $p(z)=z$ (see \cite[Section~4.2]{GrPe2011_book} for more details). It is clear that a general differential operator $\varphi(D)$ is not necessarily invertible, so that Proposition~\ref{Pro:invertible->invariant} can not always be applied. However, Godefroy and Shapiro showed that $\varphi(D)$ is a chaotic operator, and hence $\lbo(\varphi(D))$ is a dense but meager set in $H(\CC)$, as soon as $\varphi$ is a non-constant function (see \cite[Theorem~5.1]{GodSha1991}). Note that an entire function $f \in H(\CC)$ belongs to $\lbo(\varphi(D))$ if and only if there exist $k_0 \in \NN$, $\eps>0$ and $w=(w_j)_{j\in\NN} \in \ ]0,+\infty[^{\NN}$ such that
	\[
	\forall n \in \NN \quad \text{with} \quad \max_{|z| \leq k_0} \big|f(z) - [\varphi(D)^n(f)](z)\big| < \eps, \quad \text{then} \quad \max_{|z| \leq j} \big|[\varphi(D)^n(f)](z)\big| \leq w_j \quad \forall j \in \NN.
	\]
	
	Finally, given any differential operator $\varphi(D)$ we can apply Proposition~\ref{Pro:eigenvectors->lbo} obtaining that the linear span of the exponential functions $\mathcal{A}:=\lspan\{ e^{\lambda z} \ ; \ \lambda \in \CC \}$ is contained in $\lbo(\varphi(D))$. Indeed, it is well-known and not hard to check that
	\[
	\varphi(D)(e^{\lambda z}) = \varphi(\lambda) \cdot e^{\lambda z} \quad \text{ for every } \lambda \in \CC,
	\]
	so that $e^{\lambda z} \in \operatorname{Eig}(\varphi(D))$. Moreover, $\mathcal{A}$ is a dense set (see the nice proof of \cite[Lemma~2.34]{GrPe2011_book} originally from \cite[Sublemma~7]{AMar2004}), so that $\mathcal{A}$ is a dense $\varphi(D)$-invariant subalgebra of $(H(\CC),+,\cdot)$ with respect to the usual addition and pointwise product of entire functions, contained in $\lbo(\varphi(D))$ for every differential operator $\varphi(D)$. In summary, {\em differential operators} have plenty of locally bounded orbits.
\end{example}

Let us show that the set of locally bounded orbits $\lbo(T)$ is not necessarily $T$-invariant when the studied operator $T \in \Lc(X)$ is not invertible:

\begin{example}[\textbf{The backward shift on the space of all sequences}]\label{Ex:B.omega}
	Let $\omega = \KK^{\NN}$ be the space of all (real or complex) sequences endowed with the standard Fr\'echet topology of convergence in all coordinates (see for instance \cite[Example~2.2]{GrPe2011_book}). Consider the {\em backward shift operator} $B:\omega\longrightarrow\omega$, which acts as $B((x_j)_{j\in\NN}) := (x_{j+1})_{j\in\NN}$ for each $x=(x_j)_{j\in\NN} \in \omega$. It is well-known and easy to check that $B$ is chaotic, so that $\lbo(B)$ is a dense but meager set in $\omega$, and a sequence $x=(x_j)_{j\in\NN} \in \omega$ belongs to $\lbo(B)$ if and only if there exist $k_0 \in \NN$, $\eps>0$ and $w=(w_j)_{j\in\NN} \in \ ]0,+\infty[^{\NN}$ such that
	\[
	\forall n \in \NN \quad \text{with} \quad \max_{1\leq j\leq k_0} |x_j - x_{n+j}| < \eps, \quad \text{then} \quad |x_{n+j}| \leq w_j \quad \forall j \in \NN.
	\]
	
	Note that the set $\ell^{\infty}_{\KK} := \{ (x_j)_{j\in\NN} \in \KK^{\NN} \ ; \ \sup_{j\in\NN} |x_j| < +\infty \}$ of all (real or complex) bounded sequences is a dense linear subspace of $\lbo(B)$. Indeed, $\ell^{\infty}_{\KK}$ is even a dense $B$-invariant subalgebra of the Fr\'echet algebra $(\omega,+,\cdot)$ with respect to the coordinatewise addition and product of sequences, which shows that the backward shift on $\omega$ presents plenty of locally bounded orbits.
	
	Contrary to Birkhoff operators, the backward shift $B:\omega\longrightarrow\omega$ is not invertible and we claim that it does not satisfy the conclusion of Proposition~\ref{Pro:invertible->invariant}. Indeed, one way of proving that $B(\lbo(B))$ is not included in $\lbo(B)$ is the following: we construct a vector $y = (y_j)_{j\in\NN} \in \omega\setminus\lbo(B)$ such that $y_j>0$ for every $j \in \NN$ (so that $y$ is non-hypercyclic for $B$), and we consider $z = (z_j)_{j\in\NN} \in \omega$ with
	\[
	z_j = \begin{cases} -1 & \text{ if } j=1,\\ y_{j-1} & \text{ if } j>1, \end{cases}
	\]
	since then we will have $z \in \lbo(B)$ but $Bz = y \notin \lbo(B)$. We construct $y=(y_j)_{j\in\NN}$ recursively:
	\begin{enumerate}[--]
		\item \textbf{Step 1}: We start by fixing any finite word of positive numbers $(y_1,y_2,y_3,...,y_N) \in \ ]0,+\infty[^N$, which will be the first $N\geq 1$ coordinates of the final vector $y=(y_j)_{j\in\NN} \in \omega$.
		
		\item \textbf{Step 2}: We now consider the following sequence of finite words
		\begin{align*}
			&(y_1,1),\\
			&(y_1,2), \quad (y_1,y_2,2),\\
			&(y_1,3), \quad (y_1,y_2,3), \quad (y_1,y_2,y_3,3),\\
			&\quad \vdots\\
			&(y_1,N), \quad (y_1,y_2,N), \quad (y_1,y_2,y_3,N), \quad \cdots \quad (y_1,y_2,y_3,...,y_N,N),
		\end{align*}
		and we add them after the first $N$ coordinates we already had, obtaining the ``new'' finite sequence
		\[
		(y_1,y_2,y_3,...,y_N,y_{N+1},...,y_{N+\phi(N)}) := (y_1,y_2,y_3,...,y_N,\underbrace{y_1,1},\underbrace{y_1,2},\underbrace{y_1,y_2,2},...,\underbrace{y_1,y_2,y_3,...,y_N,N}),
		\]
		where we have the equality $\phi(N) = \sum_{i=1}^{N} (i+1)(N+1-i) = \frac{N^3+6N^2+5N}{6}$.
		
		\item \textbf{Step 3}: We repeat \textbf{Step 2} infinitely many times, but each time on the ``new'' and strictly longer finite sequence obtained from the previous application of \textbf{Step 2}, and we let $y=(y_j)_{j\in\NN} \in \omega$ be the final limit sequence obtained from this recursive process.
	\end{enumerate}
	Note that $y = (y_j)_{j\in\NN} \notin \lbo(B)$ since for every $k \in \NN$ we have that the finite word $(y_1,y_2,y_3,...,y_k,M)$ appears along the sequence $(y_j)_{j\in\NN}$ for every positive integer $M \in \NN$.
\end{example}

\begin{example}[\textbf{The backward shift on K\"othe spaces}]\label{Ex:B.lambda}
	Following \cite[Chapter~27]{MeisVogt1997_book} we will say that an infinite matrix $A = (a_{k,j})_{k,j\in\NN}$ of non-negative numbers is a {\em K\"othe matrix} if it satisfies:
	\begin{enumerate}[(KM1)]
		\item The inequality $a_{k,j} \leq a_{k+1,j}$ holds for all $k,j \in \NN$.
		
		\item For each $j \in \NN$ there exists some $k \in \NN$ such that $a_{k,j}>0$.
	\end{enumerate}
	Given such a matrix $A = (a_{k,j})_{k,j\in\NN}$ and $1\leq p<\infty$ the {\em K\"othe space} $\lambda^p(A)$ is defined as
	\[
	\lambda^p(A) := \left\{ x = (x_j)_{j\in\NN} \in \KK^{\NN} \ ; \ q_k(x) := \left( \sum_{j\in\NN} |x_j \cdot a_{k,j}|^p \right)^{1/p} < \infty \text{ for all } k \in \NN \right\},
	\]
	and for $p=\infty$ the {\em K\"othe space} $\lambda^{\infty}(A)$ is defined as
	\[
	\lambda^{\infty}(A) := \left\{ x = (x_j)_{j\in\NN} \in \KK^{\NN} \ ; \ r_k(x) := \sup_{j\in\NN} |x_j| \cdot a_{k,j} < \infty \text{ for all } k \in \NN \right\},
	\]
	where $(q_k)_{k\in\NN}$ and $(r_k)_{k\in\NN}$ are sequences of seminorms defining the topology of $\lambda^p(A)$ and $\lambda^{\infty}(A)$.
	
	Assume now that $A=(a_{k,j})_{k,j\in\NN}$ is a K\"othe matrix and that $p \in [1,\infty]$ is a fixed value such that the {\em backward shift operator} $B:\lambda^p(A)\longrightarrow\lambda^p(A)$, acting in $\lambda^p(A)$ as in Example~\ref{Ex:B.omega}, is well-defined and hence continuous by the Closed Graph Theorem. Using the known characterization of bounded sets for K\"othe spaces (see \cite[Lemma~27.5]{MeisVogt1997_book}) a vector $x=(x_j)_{j\in\NN} \in \lambda^p(A)$ belongs to $\lbo(B)$ if and only if there exist $k_0 \in \NN$, $\eps>0$ and $w=(w_j)_{j\in\NN} \in \ ]0,+\infty[^{\NN} \ \cap \ \lambda^{\infty}(A)$ such that
	\[
	\forall n \in \NN \quad \text{with} \quad \sum_{j\in\NN} |(x_j - x_{n+j}) a_{k_0,j}|^p < \eps^p, \quad \text{then} \quad \sum_{j\in\NN} \left( \frac{|x_{n+j}|}{w_j} \right)^p \leq 1,
	\]
	when $1\leq p< \infty$, or such that
	\[
	\forall n \in \NN \quad \text{with} \quad \sup_{j\in\NN} |x_j - x_{n+j}| \cdot a_{k_0,j} < \eps, \quad \text{then} \quad \sup_{j\in\NN} \frac{|x_{n+j}|}{w_j} \leq 1,
	\]
	when $p=\infty$. The dynamical behaviour of $B$, but also the space $\lambda^p(A)$, strongly depend on the matrix $A$, so there are not general locally bounded orbits for $B:\lambda^p(A)\longrightarrow\lambda^p(A)$. Indeed, $\lambda^p(A)$ could be a Banach space and hence $\lbo(B)=\lambda^p(A)$, but also the operator $B:\omega\longrightarrow\omega$ from Example~\ref{Ex:B.omega} is a particular case of $B:\lambda^{\infty}(A)\longrightarrow\lambda^{\infty}(A)$, precisely when $a_{k,j}=1$ for $j\leq k$ and $0$ otherwise.
\end{example}


We finish this subsection by showing that if a locally bounded orbit has a too strong recurrent behaviour then the orbit has to be bounded. This will be the case for uniformly recurrent locally bounded orbits. A vector $x \in X$ is called {\em uniformly recurrent for $T \in \Lc(X)$} if for every neighbourhood $U$ of $x$ the return set $\Nc_T(x,U) = \{ n \in \NN \ ; \ T^nx \in U \}$ has bounded gaps, that is, if for the strictly increasing sequence of integers $(n_k)_{k\in\NN}$ forming the set $\Nc_T(x,U) = \{ n_k \ ; \ k \in \NN \}$ we have that
\[
\sup_{k\in\NN} (n_{k+1} - n_k) < \infty.
\]
We will denote by $\URec(T)$ the {\em set of uniformly recurrent vectors for $T$}, and we will say that the operator $T$ is {\em uniformly recurrent} whenever the set $\URec(T)$ is dense in $X$. Uniform recurrence is a very strong notion that has sometimes been called ``almost periodicity'' (see for instance \cite{BoGrLoPe2022}). It is not hard to check that the orbit of a uniformly recurrent vector for an operator $T \in \Lc(X)$ is bounded when $X$ is a Banach space, but as shown in \cite[Example~3.3]{BoGrLoPe2022}, not necessarily bounded when $X$ is just a (non-Banach) Fr\'echet space. Let us show that local boundedness is the key in this fact:

\begin{proposition}\label{Pro:lbo->URec.bounded}
	Given an operator $T \in \Lc(X)$ acting on a Fr\'echet space $X$, the orbit of a uniformly recurrent vector $x \in \URec(T)$ is locally bounded for $T$ if and only if its orbit $\orb_T(x)$ is bounded. Hence, if the set $\URec(T) \cap \lbo(T)$ is non-meager then $T$ is power-bounded and $X=\URec(T)$.
\end{proposition}
\begin{proof}
	If there is a neighbourhood $U$ of $x$ such that $U \cap \orb_T(x)$ is a bounded set and $N \in \NN$ is the maximum gap between two consecutive elements from $\Nc_T(x,U)$, then one can check that
	\[
	\orb_T(x) \subset \bigcup_{j=0}^N T^j\left( U \cap \orb_T(x) \right),
	\]
	which are bounded sets by the continuity of $T$. Hence, if $\URec(T) \cap \lbo(T)$ is a non-meager set then $T$ is power-bounded by Banach-Steinhaus, and by \cite[Theorem~3.1]{BoGrLoPe2022} we get that $X=\URec(T)$.
\end{proof}

This last result is an extension of \cite[Corollary~3.2]{BoGrLoPe2022} that shows, in some way, what kind of (weak) boundedness is exhibited by a uniformly recurrent orbit in a Fr\'echet space (see \cite[Section~3]{BoGrLoPe2022}). We would also like to point out that Proposition~\ref{Pro:lbo->URec.bounded} does not hold for locally bounded orbits with a weaker recurrent behaviour than that of uniform recurrence:

\begin{example}
	From the Furstenberg families $\Fc \subset \Part(\NN)$ studied in the literature in the context of Linear Dynamics, the next notion of $\Fc$-recurrence slightly weaker than uniform recurrence is that of {\em frequent recurrence} as defined in the Introduction. As we have mentioned in Subsection~\ref{SubSec_3.1:applications+optimality}, this notion coincides with {\em $\Fc$-recurrence} for the family $\Fc=\Dinf$ of positive lower density sets. A standard separation result such as \cite[Lemma~9.5]{GrPe2011_book} shows that the family $\Fc=\Dinf$ has the following property:
	\[
	(*) \quad \begin{cases} \text{there exists a sequence of pairwise disjoint sets } (A_k)_{k\in\NN} \in \Fc^{\NN} \text{ such that:} \\[5pt] \text{for every } n \in A_k \text{ and every } n' \in A_{k'} \text{ with } n\neq n', \ \text{ then } \ |n-n'|\geq \max\{k,k'\}. \end{cases}
	\]
	Let us show that, given a Furstenberg family $\Fc \subset \Part(\NN)$ fulfilling $(*)$, then we can construct for $B:\omega\longrightarrow\omega$ a vector $x \in \Fc\Rec(B) \cap \lbo(B)$ such that $\orb_B(x)$ is not bounded:
	
	First of all, we may assume that $\min(A_k)>k$ for every $k \in \NN$ by taking a subsequence of $(A_k)_{k\in\NN}$ if necessary. Let $(m_s)_{s\in\NN} \in \NN^{\NN}$ be the increasing sequence of integers forming the set $\bigcup_{k\in\NN} A_k$. We now construct the vector $x=(x_j)_{j\in\NN} \in \omega$ recursively:
	\begin{enumerate}[--]
		\item \textbf{Step 1}: We start by letting $x_1=1$, and $x_j=0$ for every $1<j\leq m_1$ in case that $1<m_1$. We have fixed the first ``$m_1$'' coordinates of the final vector $x=(x_j)_{j\in\NN} \in \omega$.
		
		\item \textbf{Step 2}: Now we have that $m_1 \in A_l$ for some $l \in \NN$. By property $(*)$ and the assumptions on $(A_k)_{k\in\NN}$ we have that $m_2-m_1 \geq l$ and also that $m_1>l$, so we have no problems in letting
		\[
		(x_{m_1+1},x_{m_1+2},...,x_{m_1+l-1},x_{m_1+l}) :=
		\begin{cases}
			(x_1,x_2,x_3,...,x_{l-1},x_l) & \text{ if } l=2k+1 \text{ for some } k \geq 0, \\[5pt]
			(1,2,3,4,5,...,l-1,l) & \text{ if } l=2k \text{ for some } k \geq 1,
		\end{cases}
		\]
		and we let $x_j=0$ for every $m_1+l<j\leq m_2$ in case that $m_1+l<m_2$. We have fixed the first ``$m_2$'' coordinates of the final vector $x=(x_j)_{j\in\NN} \in \omega$.
		
		\item \textbf{Step 3}: We repeat \textbf{Step 2} infinitely many times by considering each time $m_s$ for $s\geq 2$, and using that $m_{s+1}-m_s \geq l$ as soon as $m_s \in A_l$ by $(*)$, but also that $m_s>l$ by the assumptions on $(A_k)_{k\in\NN}$. We let $x=(x_j)_{j\in\NN} \in \omega$ be the final limit sequence obtained from this recursive process.
	\end{enumerate}
	From the previous construction it is not hard to check that $x = (x_j)_{j\in\NN}$ fulfills the characterization of locally bounded orbit given in Example~\ref{Ex:B.omega} for $B:\omega\longrightarrow\omega$ with the parameters $k_0:=1$, any positive value $0<\eps<1$, and the sequence $w = (w_j)_{j\in\NN} := (j)_{j\in\NN} \in \ ]0,+\infty[^{\NN}$. However, the set $\orb_B(x)$ is not bounded since $x = (x_j)_{j\in\NN} \notin \ell^{\infty}_{\KK} = \{ (x_j)_{j\in\NN} \in \KK^{\NN} \ ; \ \sup_{j\in\NN} |x_j| < +\infty \}$. Moreover, the construction implies that for each positive integer $k \in \NN$ we have the following equality
	\[
	\max_{1\leq j\leq 2k+1} \left| \left[B^{n}(x)\right]_j - x_j \right| = 0 \quad \text{ for every } n \in A_{2k+1} \in \Fc.
	\]
	We deduce that for each neighbourhood $U$ of $x$ in $\omega$, with respect to the topology of convergence in all coordinates, the return set $\Nc_B(x,U)$ belongs to $\Fc$ and hence $x \in \Fc\Rec(B)$. We would like to mention that the families $\Fc \subset \Part(\NN)$ fulfilling $(*)$ have been called {\em hypercyclicity sets} in \cite{BesMePePu2016}. Assuming $(*)$ it is also not hard to construct an $\Fc$-hypercyclic vector for the backward shift $B:\omega\longrightarrow\omega$.
\end{example}

\subsection{Some problems}


In Example~\ref{Ex:B.omega} we show that the set $\lbo(T)$ is not necessarily $T$-invariant when the operator $T \in \Lc(X)$ is not invertible, but one can check that the inclusion $B(\Rec(B) \cap \lbo(B)) \subset \Rec(B) \cap \lbo(B)$ holds for the backward shift operator $B$ in both Examples~\ref{Ex:B.omega} and \ref{Ex:B.lambda}. This motivates our first problem: 

\begin{problem}
	Is the inclusion $T(\Rec(T) \cap \lbo(T)) \subset \Rec(T) \cap \lbo(T)$ true for every continuous linear operator $T \in \Lc(X)$? What about the inclusion $T(X\setminus\lbo(T)) \subset X\setminus\lbo(T)$?
\end{problem}

Note that if $T(X\setminus\lbo(T)) \subset X\setminus\lbo(T)$ is true, then $T(\Rec(T) \cap \lbo(T)) \subset \Rec(T) \cap \lbo(T)$ follows trivially from the following reasoning: given $x \in \Rec(T) \cap \lbo(T)$ and an \textbf{open} neighbourhood $U$ of $x$ such that $U \cap \orb_T(x)$ is bounded, then we have that $\orb_T(x) \cap U \subset \Rec(T) \cap \lbo(T)$.


An operator $S \in \Lc(Y)$ is said to be {\em quasi-conjugate} (resp.\ {\em conjugate}) to a second operator $T \in \Lc(X)$ if there exists a continuous map $J:X\longrightarrow Y$ for which $S \circ J = J \circ T$ and such that $J$ has dense range (resp.\ $J$ is an homeomorphism); and a property $P$ is said to be {\em preserved under (quasi-)conjugacy} when the following holds: if an operator $T \in \Lc(X)$ has property $P$ then every operator $S \in \Lc(Y)$ that is (quasi-)conjugate to $T$ also has property $P$. In Example~\ref{Ex:B.omega} we have shown that locally bounded orbits are not preserved under quasi-conjugacy since $B:\omega\longrightarrow\omega$ is quasi-conjugate to itself by taking $J=B=S$ and there exists $z \in \lbo(B)$ such that $Bz \notin \lbo(B)$. This motivates the following problem:

\begin{problem}
	Are locally bounded orbits preserved under conjugacy?
\end{problem}

Note that given $T \in \Lc(X)$, $S \in \Lc(Y)$ and an homeomorphism $J:X\longrightarrow Y$ with $S \circ J = J \circ T$, if $J$ is linear then the same arguments from Proposition~\ref{Pro:invertible->invariant} show that $J(\lbo(T)) = \lbo(S)$ and also that $\lbo(T)=J^{-1}(\lbo(S))$, so that the problem here is knowing what happens if $J$ is non-linear.


The following questions seem to be non-trivial and all of them are based on removing the locally bounded orbit assumption in Theorems~\ref{The:lbo->measure}, \ref{The:lbo->PF.property} and \ref{The:lbo->chaos}. Using Remark~\ref{Rem:Theorem.conditions}, we state them directly for adjoint operators acting on the dual of a (DF)-space:

\begin{problems}
	Let $(Y,\tau_Y)$ be a (DF)-space whose strong dual Fr\'echet space $(X,\tau_X):=(Y',\beta(Y',Y))$ is separable and let $T \in \Lc(X)$ be the adjoint of some linear map $S:Y\longrightarrow Y$. Then we ask: 
	\begin{enumerate}[(a)]
		\item Given a vector $x_0 \in \RRec(T) \setminus \lbo(T)$, does it follow that there exists a $T$-invariant probability measure $\mu$ on $(X,\Bi(X))$ fulfilling that $x_0 \in \supp(\mu)$?
		
		\item Given a non-empty open set $U \subset X$ and a Furstenberg family $\Fc \subset \Part(\NN)$, does the existence of a vector $x \in U \cap [b\Fc]\Rec(T) \setminus \lbo(T)$ implies the existence of $z \in U$ fulfilling that $\Nc_T(z,U) \in \Fc$?
		
		\item Given a left-invariant Furstenberg family $\Fc \subset \Part(\NN)$ for which $T$ is almost-$b\Fc$-recurrent, does it follow that $T$ is an almost-$\Fc$-recurrent operator?
		
		\item If $T$ is an $\AP_b$-hypercyclic operator, does if follow that $T$ is Devaney chaotic?
	\end{enumerate}
	What happens if $T \in \Lc(X)$ is an operator acting on a separable reflexive Fr\'echet space $(X,\tau_X)$?
\end{problems}

\section*{Funding}

The author was partially supported by the Spanish Ministerio de Ciencia, Innovaci\'on y Universidades (grant number FPU2019/04094); also by MCIN/AEI/10.13039/501100011033/FEDER, UE Projects PID2019-105011GB-I00 and PID2022-139449NB-I00; and by the Fundaci\'o Ferran Sunyer i Balaguer.

\section*{Acknowledgments}

The author would like to thank Jos\'e Bonet and Alfred Peris for their valuable advice and numerous readings of the manuscript, and to specially thank Santiago Muro for finding the reference \cite{Sucheston1964}. The author also thanks the anonymous reviewer, whose comments have considerably improved the paper.

{\footnotesize
$\ $\\

\textsc{Antoni L\'opez-Mart\'inez}: Universitat Polit\`ecnica de Val\`encia, Institut Universitari de Matem\`atica Pura i Aplicada, Edifici 8E, 4a planta, 46022 Val\`encia, Spain. e-mail: alopezmartinez@mat.upv.es
}

\end{document}